\documentclass[a4paper,english,11pt]{article}
\usepackage[T1]{fontenc}
\usepackage[utf8]{inputenc}
\usepackage[english]{babel}
\usepackage{amsmath,amsfonts,amssymb,amsthm,mathrsfs}
\usepackage{mathtools}
\usepackage{latexsym}
\usepackage{enumerate}
\usepackage{graphicx}
\graphicspath{{figures/}}
\usepackage{epstopdf}
\usepackage{float}
\usepackage{booktabs}
\usepackage{hyperref}
\usepackage[margin=10pt,font=small]{caption}
\usepackage{xcolor}
\usepackage{listing}
\usepackage[a4paper, total={6in, 8in}]{geometry}
\usepackage{matlab-prettifier}
\usepackage{authblk}

\newcommand\Tstrut{\rule{0pt}{2.6ex}}         
\newcommand\Bstrut{\rule[-0.9ex]{0pt}{0pt}}   
\newcommand{\C}{\mathbb{C}}
\newcommand{\R}{\mathbb{R}}
\newcommand{\Z}{\mathcal{Z}}
\newcommand{\N}{\mathbb{N}}
\newtheorem{theorem}{Theorem}[section]

\newtheorem{lemma}{Lemma}[section]

\theoremstyle{definition}
\newtheorem*{definition}{Definition}
\theoremstyle{definition}
\newtheorem{example}{Example}
\theoremstyle{remark}
\newtheorem*{remark}{Remark}
\providecommand{\keywords}[1]
{
  \small	
  \textbf{\textit{Keywords---}} #1
}
\title{A pseudo-spectral Strang splitting method for linear dispersive problems with transparent boundary conditions}
\author[1]{L. Einkemmer}
\author[1]{ A. Ostermann}
\author[1]{M. Residori}
\affil[1]{Department of Mathematics, University of Innsbruck}
\begin{document}
\maketitle
\begin{abstract}
The present work proposes a second-order time splitting scheme for a linear dispersive equation with a variable advection coefficient subject to transparent boundary conditions. For its spatial discretization, a dual Petrov--Galerkin method is considered which gives spectral accuracy.
The main difficulty in constructing a second-order splitting scheme in such a situation lies in the compatibility condition at the boundaries of the sub-problems. In particular, the presence of an inflow boundary condition in the advection part results in order reduction. To overcome this issue a modified Strang splitting scheme is introduced that retains second-order accuracy. For this numerical scheme a stability analysis is conducted. In addition, numerical results are shown to support the theoretical derivations.
\end{abstract}

\section{Introduction}
\label{intro}
The aim of this paper is to develop time splitting schemes in combination with transparent boundary conditions that have spectral accuracy in space. Splitting schemes are based on the divide and conquer idea; i.e. to divide the original problem into smaller sub-problems which are, hopefully, easier to solve. However, obtaining an approximation of the solution of the original problem from the solutions of the sub-problems is not always straightforward: order reductions or strong CFL conditions that destroy the convergence of the numerical scheme are known to arise, see e.g.~\cite{einkemmer14,einkemmer16,nakano19}. Furthermore, transparent boundary conditions are non-local in time and depend on the solution. Imposing them with splitting methods poses a challenge in the derivation of stable numerical schemes of order higher than one.

In this paper, we show that it is possible to construct a second order splitting scheme that performs well, in the context outlined above, and can be implemented efficiently. In particular, we show that the proposed numerical method is stable independent of the space grid spacing (i.e.~no CFL type condition is needed). We focus our attention on a linearised version of the Korteweg--de Vries equation
\begin{equation}
\label{eq1}
\begin{cases}
\partial_t u (t,x)+ g(x) \partial_x u(t,x) + \partial_x^3 u(t,x) = 0, \quad (t,x)\in [0,T]\times \R,\\
 u(0,x) = u^0(x),
\end{cases}
\end{equation}
where $T>0$. The same ideas, however, can be applied to a more general set of linear partial differential equations with variable coefficients.
Note that the partial differential equation~\eqref{eq1}, despite being linear, finds many applications in a physical context. For example, it is used to model long waves in shallow water over an uneven bottom, see e.g.~\cite{kakutani71,whitham74}.

The goal of this work is to design a splitting scheme that is second order in time with spectral accuracy in space. This paper can be seen as an extension to~\cite{residori20}, where a splitting scheme of order one in time and spectral accuracy in space is presented. When solving~\eqref{eq1} one of the main difficulties one has to face is the unbounded domain $\R$. Numerical simulations typically consider a finite domain that leads to boundary conditions. Our goal is to design a numerical scheme that retains the same dynamics as the original problem~\eqref{eq1}, but on a finite domain. This can be achieved by imposing transparent boundary conditions. The advantage of such boundary conditions is the zero-reflection property of the solution at the boundaries. Further, the solution can leave the finite domain and re-enter at a later time without any loss of information. On the downside, transparent boundary conditions are non-local in time (and space for two and three-dimensional problems), therefore, they become expensive for long time simulations. In particular, memory requirements grow proportionally with the number of time steps. While it is still possible to employ them in 1D, the multidimensional cases become impracticable. A remedy is to approximate transparent boundary conditions and obtain so-called absorbing boundary conditions. In this way, information at the boundaries is lost, but memory requirements remain constant. A lot of work has been done for the Schr\"odinger equation in recent years, see \cite{antoine08,arnold03,bertoli17} and references therein. For third-order problems, we refer the reader to~\cite{besse16,besse16a,residori20,zheng08} and references therein. 

In the present case, the third derivative in space renders any explicit integrator extremely expensive. Therefore, an implicit scheme should be implemented. While coupling an implicit time discretization with a spectral space discretization yields banded matrices for constant advection, they lead to full matrices if $g$ varies in space. We therefore employ a time-splitting approach in order to separate the advection problem from the dispersive problem. Operator splitting methods for dispersive problems have been employed and studied before, we refer the reader to~\cite{einkemmer15,einkemmer18,residori20,holden11}. For splitting method with absorbing boundary conditions we cite the work~\cite{bertoli17}. Splitting methods allow us to design specific solvers for the variable coefficient problem. For example, \cite{shen07} uses a technique based on preconditioning. However, a direct splitting of~\eqref{eq1} is not advisable. The problem of separating the advection equation is the potential requirement of inflow conditions at the boundaries. The actual inflow, however is unknown and should be estimated for example by extrapolation methods. This leads to instabilities when spectral methods are applied unless a very restrictive CFL condition is satisfied. The idea to overcome this problem is to perform a modified splitting that allows us to treat the advection problem without prescribing any inflow condition. The boundary conditions are transferred to the dispersive problem only. In this case, we can compute the values we need with the help of the $\Z$-transform, as has been done for a constant coefficient dispersive problem in~\cite{besse16, besse16a}.
Another popular technique to avoid reflections at the boundaries is the perfectly matched layer method (PML). This method has been introduced in~\cite{berenger94} for Maxwell's equations. Subsequently, it has been adapted to the Schr\"odinger equation~\cite{zheng07} and very recently  a general PML approach in combination with pseudo-spectral methods has been proposed in~\cite{antoine20}. To the best of our knowledge a PML method for a linearised Korteweg--de Vries equation is currently not available. 

The paper is organized as follows. In Section 2 we derive the semi-discrete scheme, discrete in time and continuous in space, by applying the Strang splitting method. In Section~3 we impose transparent boundary conditions for the scheme derived in Section~2. In particular, we determine the proper values of the numerical solution at the boundaries with the help of the $\Z$-transform. The stability of the resulting numerical method is then analyzed in section 4. In Section 5 we describe a pseudo-spectral method for the spatial discretization which takes the transparent boundary conditions into account. Finally, in Section~6 we present  numerical results that illustrate the theory.
\section{Time discretization: modified splitting approach}
\label{timedisc}
In this section we derive a semi-discrete scheme by applying the Strang-splitting method to problem~\eqref{eq1} restricted to a finite interval $[a,b]$, where $a<b$. Inspired by the ideas in~\cite{residori20}, we perform a time splitting in order to separate the advection problem $\partial_t u(t,x) + g(x) \partial_x u(t,x) = 0$ from the dispersive problem $\partial_t u(t,x) + \partial_x^3 u(t,x) =0$.
In the following, for brevity, time and space dependence for the unknown $u=u(t,x)$ are omitted.

In Section~\ref{canspl} we present the canonical splitting of~\eqref{eq1}. This approach illustrates the difficulty to prescribe the inflow condition to the advection equation. In Section~\ref{modspl} we then propose the modified splitting and show how this problem can be avoided.

\subsection{Canonical splitting}
\label{canspl}
Before applying any splitting, a preliminary analysis shows us that the inflow conditions to the advection problem depend on the sign of $g(x)$ at $x=a$ and $x=b$. We summarise in Table~\ref{tab0} the four possible outcomes.
\begin{table}[b]
\small
\begin{center}
\begin{tabular}{c | c c}
 & $ g(a)>0 $ &  $g(a)\leq 0$ \Bstrut\\
\hline
$g(b)<0$ &  $a,b$ & $b$ \Tstrut\\
$g(b)\geq 0$  &  $a$ & -- \\
\end{tabular}
\caption{This table summarises at which boundary points $\{a,b\}$ the inflow condition needs  to be prescribed for the advection problem, depending on the sign of $g(x)$ at the boundaries.}
\label{tab0}
\end{center}
\end{table}
For this presentation, we restrict our attention to $g(x) > 0$ for $x\in [a,b]$. This setting requires an inflow condition at $x=a$. Let  $M\in \N$, $M>0$ be the number of time steps, $\tau = T/M$ the step size and $t^m = m\tau$, $k=0,\dots,M$. We apply the Strang splitting method to~\eqref{eq1}, which results in the two sub-problems
\begin{subequations}
\begin{align}
\label{eq41a}
& \begin{cases}
\partial_t v + \partial_x^3 v = 0,\\
 v(0,x) = v^0(x),\\
\end{cases} \\
& \label{eq41b}
\begin{cases}
\partial_t w + g\partial_x w  = 0,\\
 w(0,x) = w^0(x).\\
\end{cases}
\end{align}
\end{subequations}
Let $\varphi_{t}^{[1]}$ be the flow of \eqref{eq41a} and let $\varphi_{t}^{[2]}$ be the flow of \eqref{eq41b}. Let $u(t,x)$ be the solution of~\eqref{eq1} at time $t$. Then, the solution to~\eqref{eq1} at time $t + \tau$ is approximated by the Strang splitting
\begin{equation}
\label{eq42}
u(t+\tau, \cdot) \approx \varphi_{\frac{\tau}{2}}^{[1]} \circ \varphi_{\tau}^{[2]} \circ \varphi_{\frac{\tau}{2}}^{[1]}\left( u(t,\cdot)\right).
\end{equation}
In order to get a numerical scheme, we apply the Peaceman--Rachford scheme to \eqref{eq42}. This consists in computing the first flow $\varphi_{\frac{\tau}{2}}^{[1]}$ by the explicit Euler method, the middle flow $\varphi_{\tau}^{[2]}$ by the Crank--Nicolson method and the last flow by the implicit Euler method. Let $u^m(x) = u(t^m,x)$. Then, we get
\begin{align}
\label{eq43}
u^* &= \left(I-\frac{\tau}{2}\partial_x^3\right)u^m, \\
\label{eq44}
\left(I+\frac{\tau}{2}g\partial_x\right)u^{m+1/2} & = \left(I-\frac{\tau}{2}g\partial_x\right)u^*,\\
\label{eq45}
\left(I+\frac{\tau}{2}\partial_x^3\right)u^{m+1} &= u^{m+1/2}.
\end{align}
The latter numerical scheme is known to be second-order in time due to its symmetry.
 Notice that Equation~\eqref{eq44} is a time approximation of
\begin{equation}
\label{eq40}
\begin{cases}
\partial_t u + g\partial_x u = 0,\quad (t,x)\in[0,\tau]\times [a,b],\\
u(0,x) = u^*(x),\\
u(t,a) = f(t).\\
\end{cases}
\end{equation}
The function $f(t)$ encodes the inflow condition at $x=a$. For $t\in(0,\tau]$ the inflow condition is unknown. It can be approximated by extrapolation methods which typically leads to instabilities. The idea to overcome this problem is to formulate the advection problem without any inflow condition. For this purpose we introduce next a modified splitting.

\subsection{Modified splitting}
\label{modspl}
Based on the observations in Section~\ref{canspl}, we rewrite the governing equation in~\eqref{eq1} as follows
\begin{align*}
& \partial_t u + g(x) \partial_x u + \partial_x^3 u =
 \partial_t u + \left(g(x)-p_g(x) + p_g(x)\right) \partial_x u + \partial_x^3 u,
\end{align*}
where $p_g(x)$ is the line connecting the points $\left(a,g(a)\right)$ and $(b,g(b))$. We now apply a splitting method that results in the two sub-problems
\begin{subequations}
\begin{align}
\label{eq2a}
& \begin{cases}
\partial_t v + p_g(x)\partial_x v + \partial_x^3 v = 0,\\
 v(0,x) = v^0(x),\\
\end{cases} \\
& \label{eq2b}
\begin{cases}
\partial_t w + \left(g(x)-p_g(x)\right)\partial_x w  = 0,\\
 w(0,x) = w^0(x).\\
\end{cases}
\end{align}
\end{subequations}
Let $\varphi_{t}^{[1]}$ be the flow of \eqref{eq2a} and let $\varphi_{t}^{[2]}$ be the flow of \eqref{eq2b}. Let $u(t,x)$ be the solution of~\eqref{eq1} at time $t$. The solution to~\eqref{eq1} at time $t + \tau$ is then approximated by the Strang splitting
\begin{equation}
\label{eq3}
u(t+\tau, \cdot) \approx \varphi_{\frac{\tau}{2}}^{[1]} \circ \varphi_{\tau}^{[2]} \circ \varphi_{\frac{\tau}{2}}^{[1]}\left( u(t,\cdot)\right).
\end{equation}
By applying the Peaceman-Rachford scheme to \eqref{eq3}, we get
\begin{align}
\label{eq4}
u^* &= \left(I-\frac{\tau}{2}p_g(x)\partial_x-\frac{\tau}{2}\partial_x^3\right)u^m, \\
\label{eq5}
\left(I+\frac{\tau}{2}g^*\partial_x\right)u^{m+1/2} & = \left(I-\frac{\tau}{2}g^*\partial_x\right)u^*,\\
\label{eq6}
\left(I+\frac{\tau}{2}p_g\partial_x + \frac{\tau}{2}\partial_x^3\right)u^{m+1} &= u^{m+1/2},
\end{align}
where $g^*(x) = g(x) - p_g(x)$.
Notice that $g^*(a) = g^*(b) = 0$. This means that no inflow or outflow condition needs to be prescribed to Equation~\eqref{eq5}. The modified splitting allows us to solve the advection equation only for the interior points, i.e. $x\in(a,b)$.

\begin{remark}
    Notice that both problems~\eqref{eq2a}, \eqref{eq2b} have a variable coefficient advection. However, as shown in Section~\ref{spacedisc} the matrix associated to the space discretization of Problem~\eqref{eq2a}, despite the space dependent coefficient $p_g$, is still banded. This is a property of the spectral space discretization that we employ.
\end{remark}

\section{Discrete transparent boundary conditions}
\label{dtbcs}
When it comes to numerical simulations,  a finite spatial domain is typically considered. Problem~\eqref{eq1} is then transformed into the following boundary value problem
\begin{equation}
\label{eq7}
\begin{cases}
\partial_t u + g \partial_x u + \partial_x^3 u = 0, \quad (t,x)\in [0,T]\times (a,b),\\
 u(0,x) = u^0(x),\\
 u(t,x)|_{x=a} = u(t,a),\\
 u(t,x)|_{x=b} = u(t,b),\\
 \partial_x u(t,x)|_{x=b} = \partial_x u(t,b).
\end{cases}
\end{equation}
Due to the third order dispersion term, three boundary conditions are required. In particular, depending on the sign of the dispersion coefficient, we have either two boundary conditions at the right boundary and one at the left boundary or vice-versa. In this work we consider a positive dispersion coefficient.
We assume $g(x)$ constant for $x\in\R\setminus [a,b]$ and that $u^0(x)$ is a smooth initial value with compact support in $[a,b]$.
Transparent boundary conditions are established by considering~\eqref{eq7} on the complementary unbounded domain $\R\setminus(a,b)$. Let $g_{a,b}$ be the values of $g(x)$ in $(-\infty, a]$ and $[b,\infty)$, respectively. In the interval $(-\infty, a]$ we consider the problem
\begin{equation}
\label{eq8}
\begin{cases}
\partial_t u + g_a \partial_x u + \partial_x^3 u = 0, \quad (t,x)\in [0,T]\times (-\infty, a),\\
 u(0,x) = 0,\\
 u(t,x)|_{x=a} = u(t,a),\\
 \lim_{x\to-\infty} u(t,x) = 0,\\
\end{cases}
\end{equation}
whereas in the interval $[b,\infty)$ we consider the problem
\begin{equation}
\label{eq9}
\begin{cases}
\partial_t u + g_b \partial_x u + \partial_x^3 u = 0, \quad (t,x)\in [0,T]\times (b,\infty),\\
 u(0,x) = 0,\\
 u(t,x)|_{x=b} = u(t,b),\\
 \lim_{x\to+\infty} u(t,x) = 0.
\end{cases}
\end{equation}
The initial value $u(0,x)$ is set to $0$ because $u^0(x)$ has compact support in $[a,b]$. The boundary conditions at $x\to\pm \infty$ are set to $0$ because we ask for $u\in L^2(\R)$. Therefore, the solution $u$ must decay for $x\to\pm \infty$. We focus our attention on~\eqref{eq8} and  impose discrete transparent boundary conditions at $x=a$. A similar procedure can be applied to~\eqref{eq9}.

The mathematical tool we employ in order to impose discrete transparent boundary conditions to~\eqref{eq7} is the $\Z$-transform.  We recall the definition and the main properties of the $\Z$-transform, which are used extensively in this section. For more details we refer the reader to~\cite{arnold03}. The $\Z$-transform requires an \emph{equidistant} time discretization.
Given a sequence $\mathbf{u} = \{u^l\}_l$, its $\Z$-transform is defined by
\begin{equation}
\label{eq100}
\hat{u}(z):=\Z\left(\mathbf{u}\right)(z) = \sum_{l=0}^{\infty} z^{-l} u^l,\quad z\in\C,\,|z|>\rho\geq 1,
\end{equation}
where $\rho$ is the radius of convergence of the series.
The following properties hold
\begin{itemize}
\item[ ] \emph{Linearity}: for $\alpha,\beta\in\mathbb{R}$, $\mathcal{Z}(\alpha\mathbf{u}+\beta\mathbf{v})(z) = \alpha\hat{u}(z)+ \beta\hat{v}(z)$;
\item[ ] \emph{Time advance}: for $k>0$, $\mathcal{Z}(\{u^{l+k}\}_{l\geq0})(z) = z^k\hat{u}(z)-z^k\sum_{l=0}^{k-1}z^{-l}u^l$;
\item[ ] \emph{Convolution}: $\mathcal{Z}\big(\mathbf{u} *_d \mathbf{v}\big)(z) = \hat{u}(z)\hat{v}(z)$;
\end{itemize}
where $*_d$ denotes the discrete convolution
\[
(\mathbf{u} *_d \mathbf{v})^m:= \sum_{j=0}^m u^jv^{m-j},\quad m\geq 0.
\]

\begin{remark}
The Peaceman--Rachford scheme given in \eqref{eq4}--\eqref{eq6} reduces to a Crank--Nicolson scheme outside the computational domain $[a,b]$. Therefore, discrete transparent boundary conditions are derived discretizing~\eqref{eq8} by the Crank--Nicolson method.
\end{remark}

Discretizing \eqref{eq8} by the Crank--Nicolson method, gives
\begin{equation}
\label{eq10}
\left(I + \frac{\tau g_a}{2}\partial_x + \frac{\tau}{2}\partial_x^3\right)u^{m+1}(x) =
\left(I - \frac{\tau g_a}{2}\partial_x - \frac{\tau}{2}\partial_x^3\right)u^{m}(x),\quad u^0(x) = 0.
\end{equation}

Let $\mathbf{u}(x) = \{u^k(x)\}_k$ be the time sequence ($x$ plays the role of a parameter) associated to the Crank--Nicolson scheme~\eqref{eq10}. Then its $\Z$-transform is given by
\[
\hat{u}(x,z):=\Z\{\mathbf{u}(x)\}(z) = \sum_{l=0}^{\infty} u^l(x) z^{-l}.
\]
Taking the $\Z$-transform of~\eqref{eq10} gives
\begin{equation}
\label{eq11}
z\left(I + \frac{\tau g_a}{2}\partial_x + \frac{\tau}{2}\partial_x^3\right)\hat{u}(x) =
\left(I - \frac{\tau g_a}{2}\partial_x - \frac{\tau}{2}\partial_x^3\right)\hat{u}(x),\quad x\in (-\infty,a],
\end{equation}
where we used the time advance property of the $\Z$-transform and $u^0(x) =0$. In particular, \eqref{eq11} is an ODE in the variable $x$. It can be solved by using the ansatz
\[
\hat{u}(x,z) = c_1(z)\mathrm{e}^{\lambda_1(z)x} + c_2(z)\mathrm{e}^{\lambda_2(z)x} + c_3(z)\mathrm{e}^{\lambda_3(z)x},
\]
where $\lambda_i$, $i=1,2,3$ are the roots of the characteristic polynomial associated to \eqref{eq11}:
\[
\lambda^3 + g_a\lambda + \frac{2}{\tau}\frac{1-z^{-1}}{1+z^{-1}}.
\]
The roots $\lambda_i$ can be ordered such that $\mathrm{Re}\, \lambda_{1}<0$ and $\mathrm{Re}\, \lambda_{2,3}>0$, see \cite{besse16}. By the decay condition  $\hat{u}(x,z)\to 0$ for $x\to-\infty$, we obtain $c_1(z) = 0$ and
\begin{align}
\label{eq12}
\hat{u}(x,z) &= c_2(z) \mathrm{e}^{\lambda_2(z)x} + c_3(z)\mathrm{e}^{\lambda_3(z)x},\quad x\in(-\infty,a].
\end{align}
Since $c_2$ and $c_3$ are unknown, the way to compute the discrete transparent boundary conditions is to make use of the derivatives of $\hat{u}$ to derive an implicit formulation. Computing the first and second derivative of $\hat{u}$ gives
\begin{align*}
\partial_x\hat{u}(x,z) &= \lambda_2(z)c_2(z)\mathrm{e}^{\lambda_2(z)x} +  \lambda_3(z)c_3(z)\mathrm{e}^{\lambda_3(z)x},\\
\partial_x^2\hat{u}(x,z) &= \lambda^2_2(z)c_2(z)\mathrm{e}^{\lambda_2(z)x} +  \lambda^2_3(z)c_3(z)\mathrm{e}^{\lambda_3(z)x}.\\
\end{align*}
We then have
\begin{equation}
\label{eq13}
\partial_x^2\hat{u}(x) = \left(\lambda_2+\lambda_3\right)\partial_x\hat{u}(x) - \lambda_2\lambda_3 \hat{u}(x).
\end{equation}
In the latter equation the $z$ dependence is omitted. The roots $\lambda_i$, $i=1,2,3$ satisfy
\begin{align*}
\lambda_1 + \lambda_2 + \lambda_3 &= 0,\\
\lambda_1\lambda_2 + \lambda_1\lambda_3 + \lambda_2\lambda_3 &= g_a.
\end{align*}
This allows us to rewrite Equation~\eqref{eq13} in terms of the root $\lambda_1$ to obtain
\begin{equation}
\label{eq13a}
\partial_x^2\hat{u}(x)  + \lambda_1\partial_x\hat{u}(x) + (g_a + \lambda_1^2) \hat{u}(x) = 0.
\end{equation}
We can finally determine the value of $u^{m+1}(a)$ by evaluating \eqref{eq13a} at $x=a$ and taking the inverse $\Z$-transform. Let
\[
\mathbf{Y}_1 = \Z^{-1}\left(z\mapsto \lambda_1(z)\right)\quad \text{and}\quad \mathbf{Y}_2 = \Z^{-1}\left(z\mapsto\lambda_1^2(z)\right),
\]
then
\begin{equation}
\label{eq13b}
\partial_x^2u^{m+1}(a) + \left(\mathbf{Y}_1*_d \partial_x\mathbf{u}(a)\right)^{m+1} + \left(\mathbf{Y}_2 *_d \mathbf{u}(a)\right)^{m+1} + g_a u^{m+1}(a) = 0,
\end{equation}
where we used the convolution property of the $\Z$-transform.
We remark that to compute $u^{m+1}(a)$ we need to know $\partial_x u^{m+1}(a)$ and $\partial_x^2 u^{m+1}(a)$. Similarly, for problem \eqref{eq9}, we obtain
\begin{align}
\label{eq13c}
\partial_x u^{m+1}(b) - \left(\mathbf{Y}_3*_d \mathbf{u}(b)\right)^{m+1} &= 0,\\
\label{eq13d}
\partial_x^2 u^{m+1}(b)  - \left(\mathbf{Y}_4*_d \mathbf{u}(b)\right)^{m+1} &= 0,
\end{align}
where
\[
\mathbf{Y}_3 = \Z^{-1}\left(z\mapsto\sigma_1(z)\right)\quad \text{and}\quad \mathbf{Y}_4 = \Z^{-1}\left(z\mapsto\sigma_1^2(z)\right)
\]
with $\sigma_1$ root of
\[
\sigma^3 + g_b \sigma + \frac{2}{\tau}\frac{1-z^{-1}}{1+z^{-1}},\quad \text{Re}\,\sigma_1 < 0.
\]
The time discrete numerical scheme to problem \eqref{eq7} becomes (for $0\leq m\leq M-1$)
\begin{subequations}
\begin{align}
 u^* &= \left(I-\frac{\tau}{2}p_g(x)\partial_x-\frac{\tau}{2}\partial_x^3\right)u^m, \label{eq14a} \\
 \left(I+\frac{\tau}{2}g^*\partial_x\right)u^{m+1/2} &= \left(I-\frac{\tau}{2}g^*\partial_x
\right)u^*,\label{eq14b}\\
 \left(I+\frac{\tau}{2}p_g(x)\partial_x + \frac{\tau}{2}\partial_x^3\right)u^{m+1} &= u^{m+1/2},\label{eq14c}\\
 u(0,x) &= u^0(x),\label{eq14d}\\
 \partial_x^2 u^{m+1}(a)  + Y_1^0\partial_x u(a)^{m+1} + (g_a + Y_2^0) u(a)^{m+1} &= h_1^{m+1},\label{eq14e}\\
 \partial_x u^{m+1}(b) - Y_3^0u(b)^{m+1} &= h_2^{m+1},\label{eq14f}\\
 \partial_x^2 u^{m+1}(b)  - Y_4^0 u(b)^{m+1} &= h_3^{m+1}\label{eq14g},
\end{align}
\end{subequations}
where
\[
\begin{split}
h_1^{m+1} &=\sum_{k=1}^{m+1} Y_1^k \partial_x u^{m+1-k}(a) + Y_2^k u^{m+1-k}(a),  \\
h_2^{m+1} &= \sum_{k=1}^{m+1} Y_3^k u^{m+1-k}(b),\\
 h_3^{m+1} &= \sum_{k=1}^{m+1} Y_4^k u^{m+1-k}(b).
\end{split}
\]
Equations \eqref{eq14a}--\eqref{eq14c} are the Peaceman--Rachford scheme. Equation \eqref{eq14d} is the initial data and Equations \eqref{eq14e}--\eqref{eq14g} are the discrete transparent boundary conditions.

\begin{remark}[Computation of $\mathbf{Y}_j$]
The quantities $\mathbf{Y}_j$, $j=1,\dots, 4$ are given by the inverse $\Z$-transform through Cauchy's integral formula
\[
\begin{split}
\mathbf{Y}_j^m &= \frac{1}{2\pi\mathrm{i}}\oint_{S_r}\lambda_1^j(z)z^{m-1}\mathrm{d}\,z,\quad j=1,2,\\
\mathbf{Y}_{j+2}^m &= \frac{1}{2\pi\mathrm{i}}\oint_{S_r}\sigma_1^{j}(z)z^{m-1}\mathrm{d}\,z,\quad j=1,2,
\end{split}
\]
where $S_r$ is a circle with center $0$ and radius $r>\rho$, where $\rho$ is the radius of convergence in~\eqref{eq100}. An exact evaluation of the contour integrals might be too complicated or infeasible. Therefore, we employ a numerical procedure in order to approximate these quantities. In this work we use the algorithm described in~\cite[Sec. 2.3]{fang18}, which results in stable and accurate results.
\end{remark}

\section{Stability of the semi-discrete scheme}
For this section it is convenient to adopt a more compact notation. Thus, we write $D_3 = p_g\partial_x + \partial_x^3$ and $D = g^*\partial_x$. Then, the Peaceman--Rachford scheme \eqref{eq14a}--\eqref{eq14c} becomes
\[
\begin{split}
u^* &= \left(I-\frac{\tau}{2}D_3\right)u^m,\\
\left(I+\frac{\tau}{2}D\right)u^{m+1/2} &= \left(I-\frac{\tau}{2}D\right)u^*,\\
\left(I+\frac{\tau}{2}D_3\right)u^{m+1} &= u^{m+1/2}
\end{split}
\]
for $m=0,\dots ,M-1$. The scheme can be rewritten separating the first step, i.e. when $m=0$, as follows:
\[
\begin{split}
y^0 &= \left(I-\frac{\tau}{2}D_3\right)u^0,\\
\left(I+\frac{\tau}{2}D_3\right) \left(I-\frac{\tau}{2}D_3\right)^{-1} y^{m+1} &=  \left(I+\frac{\tau}{2}D\right)^{-1}\left(I-\frac{\tau}{2}D\right)y^m,\quad  0\leq m\leq M-2,\\
\left(I+\frac{\tau}{2}D\right)u^{M-1/2} &= \left(I-\frac{\tau}{2}D\right)y^{M-1},\\
\left(I+\frac{\tau}{2}D_3\right)u^{M} &= u^{M-1/2}.
\end{split}
\]
Using the commutativity between $I+\frac{\tau}{2}D_3$ and $I-\frac{\tau}{2}D_3$ leads to
\begin{equation}
\label{eq50}
\left(I-\frac{\tau}{2}D_3\right)^{-1} \left(I+\frac{\tau}{2}D_3\right) y^{m+1} = \left(I+\frac{\tau}{2}D\right)^{-1}\left(I-\frac{\tau}{2}D\right)y^m.
\end{equation}
We now show that the semi-discrete numerical scheme~\eqref{eq50} is stable. The proof follows a similar approach as in~\cite{fang18}.
\begin{theorem}[Stability]
The semi-discrete numerical scheme~\eqref{eq50} is stable if $\partial_x g^*\in L^{\infty}(a,b)$ and $\tau < 4/\lVert\partial_x g^*\rVert_{\infty} $.
\end{theorem}
\begin{proof}
Let $(\cdot,\cdot)$ be the usual inner product on $L^2(a,b)$ and $\lVert\cdot\rVert$ the induced norm.
We define $w := \left(I+\frac{\tau}{2}D\right)^{-1}\left(I-\frac{\tau}{2}D\right)y^m$. Then
\[
\left(I+\frac{\tau}{2}D_3\right) y^{m+1}  = \left(I-\frac{\tau}{2}D_3\right)w.
\]
Applying the inner product with $y^{m+1} + w$ gives
\[
(y^{m+1},y^{m+1} + w) + \frac{\tau}{2}(D_3\, y^{m+1},y^{m+1}+w) = (w,y^{m+1}+w)-\frac{\tau}{2}(D_3\,w,y^{m+1}+w)
\]
or equivalently
\[
\lVert y^{m+1}\rVert^2 - \lVert w\rVert^2 = -\frac{\tau}{2} \left(D_3\,(y^{m+1}+w),y^{m+1}+w\right).
\]
Integrating the right-hand side by parts gives
\begin{multline}
\label{eq54}
\lVert y^{m+1}\rVert^2 - \lVert w\rVert^2 \\= -\frac{\tau}{2}\left[\partial_x^2(y^{m+1}+w)\cdot (y^{m+1} + w) -\frac{1}{2}\left(\partial_x(y^{m+1} + w\right))^2 + \frac{1}{2}p_g(y^{m+1}+w)^2\right]_{x=a}^{x=b}\\ + \frac{\tau}{4}\left(\partial_x p_g\right)\cdot\lVert y^{m+1}  + w\rVert^2.
\end{multline}
Notice that $\partial_x p_g$ is constant since $p_g$ is a polynomial of degree 1. In order to complete the proof, a bound for $\lVert w\rVert^2$ is needed. By definition of $w$, we have
\begin{equation}
\label{eq51}
\left(I+\frac{\tau}{2}D\right) w = \left(I-\frac{\tau}{2}D\right)y^m.
\end{equation}
Taking the inner product with $w+y^m$ gives
\[
\lVert w\rVert^2-\lVert y^m\rVert^2 = -\frac{\tau}{2}(D(w+y^m),w+y^m).
\]
Integrating by parts and using the fact that $g^*(a) = g^*(b) = 0$ gives
\begin{equation*}
\label{eq53}
\begin{split}
\lVert w\rVert^2-\lVert y^m\rVert^2 &= \frac{\tau}{4}\left((w+y^m)^2,\partial_x g^*\right)\\
&\leq \frac{\tau}{4}\lVert \partial_x g^*\rVert_{\infty}\left(\lVert w\rVert^2 +\lVert y^m\rVert^2\right).
\end{split}
\end{equation*}
    Using the hypothesis $\tau<4/\lVert \partial_x g^*\rVert_{\infty}$ leads to
\begin{equation}
\label{eq53b}
\lVert w\rVert^2 \leq \frac{1 + \frac{\tau}{4}\lVert \partial_x g^*\rVert_{\infty}}{1 - \frac{\tau}{4}\lVert \partial_x g^*\rVert_{\infty}}\lVert y^m\rVert^2.
\end{equation}
Combining~\eqref{eq54} with~\eqref{eq53b} gives the bound
\begin{equation}
\label{eq55}
\left( 1 -\frac{\tau}{4}\lvert\partial_x p_g\rvert\right)\lVert y^{m+1}\rVert^2 \leq B^m + \left( 1 +\frac{\tau}{4}\lvert\partial_x p_g\rvert\right)\frac{1 + \frac{\tau}{4}\lVert \partial_x g^*\rVert_{\infty}}{1- \frac{\tau}{4}\lVert \partial_x g^*\rVert_{\infty}}\lVert y^m\rVert^2,
\end{equation}
where
\[
B^m = -\tau\left[2\partial_x^2(u^{m+1})\cdot u^{m+1} -\left(\partial_x u^{m+1}\right)^2 + g\cdot(u^{m+1})^2\right]_{x=a}^{x=b}.
\]
In the definition of $B^m$ we used  $p_g(a) = g(a)$, $p_g(b) = g(b)$ and
\[
y^{m+1} + w = \left(I-\frac{\tau}{2}D_3\right)u^{m+1} + \left(I+\frac{\tau}{2}D_3\right)u^{m+1} = 2u^{m+1}.
\]
Multiplying both sides of~\eqref{eq55} by $1- \frac{\tau}{4}\lVert \partial_x g^*\rVert_{\infty}$ and taking the sum over $m$ gives
\[
\lVert y^{M}\rVert^2 - \lVert y^0\rVert^2\leq c_1\sum_{m=0}^{M-1} B^m + c_2\sum_{m=0}^{M-1} \left(\lVert y^m\rVert^2 + \lVert y^{m+1}\rVert^2\right)
\]
with
\[
c_1 = \frac{\left(1- \frac{\tau}{4}\lVert \partial_x g^*\rVert_{\infty}\right)}{1 + \frac{\tau^2}{16}\lvert\partial_x p_g\rvert\cdot\lVert\partial_x g^*\rVert_{\infty}}\geq 0,\quad
c_2 = \frac{\frac{\tau}{4}\left(\lvert \partial_x p_g\rvert + \lVert\partial_x g^*\rVert_{\infty}\right)}{ 1 + \frac{\tau^2}{16}\lvert \partial_x p_g\rvert\cdot\lVert\partial_x g^*\rVert_{\infty}}\geq 0.
\]
By Lemma~\ref{lemma1} the quantity $\sum B^m$ is negative. Therefore,
\[
\lVert y^{M}\rVert^2 - \lVert y^0\rVert^2\leq c_2\sum_{m=0}^{M-1} \left(\lVert y^m\rVert^2 + \lVert y^{m+1}\rVert^2\right)
\]
and stability follows by Gronwall's inequality since $c_2 =\mathcal{O}(\tau)$.
\end{proof}
\begin{lemma}
\label{lemma1}
It holds $\sum_{m=0}^{M-1} B^m \leq 0$.
\end{lemma}
\begin{proof}
Consider
\[
\sum_{m =0}^{M-1} B^m = \tau (B^M_a - B^M_b),
\]
where
\[
B^M_a := \sum_{m=0}^{M-1} 2\partial_x^2u^{m+1}(a)\cdot u^{m+1}(a) -\left(\partial_xu^{m+1}(a)\right)^2 + g(a)\,(u^{m+1}(a))^2
\]
and
\[
B^M_b :=\sum_{m=0}^{M-1}  2\partial_x^2u^{m+1}(b)\cdot u^{m+1}(b) -\left(\partial_x u^{m+1}(b)\right)^2 + g(b)\,(u^{m+1}(b))^2.
\]
Inserting the discrete transparent boundary conditions~\eqref{eq13b}--\eqref{eq13d} in $B^M_a,$ $B^M_b$ gives
\[
\begin{split}
B^M_a & = -\sum_{m=0}^{M-1} \left(2\left((\mathbf{Y_1}*_d\partial_x\mathbf{u}(a))^{m+1} +(\mathbf{Y_2}*_d\mathbf{u}(a))^{m+1}  + \frac{g(a)}{2} u^{m+1}(a)\right) u^{m+1}(a) +\left(\partial_x u^{m+1}(a)\right)^2\right) ,\\
B^M_b &= \sum_{m=0}^{M-1} \left(2(\mathbf{Y_4}*_d\mathbf{u}(b))^{m+1} u^{m+1}(b) -\left((\mathbf{Y_3}*_d\mathbf{u}(b))^{m+1}\right)^2 + g(b)\,(u^{m+1}(b))^2\right).
\end{split}
\]
Let us extend  the sequences $B^M_a,$ $B^M_b$ to infinity sequences by zero and apply Parseval's identity
\[
\sum_{m=-\infty}^{\infty} v_1^m\cdot \bar{v}_2^m = \frac{1}{2\pi}\int_0^{2\pi} \Z(v_1) (z)\cdot \overline{\Z(v_2)}(z)\Big|_{z = \mathrm{e}^{\mathrm{i}\theta}}\mathrm{d}\theta.
\]
We obtain
\begin{equation}
\label{eq60}
\begin{split}
B^M_a &= \frac{1}{2\pi} \int_0^{2\pi} |z|^2\left\{ -\left(2\lambda_1^2(z) + g(a)\right)|\hat{u}(a,z)|^2-\lvert\partial_x\hat{u}(a,z)\rvert^2 - 2\lambda_1(z)\partial_x\hat{u}(a,z) \overline{\hat{u}(a,z)}\right\}\Big|_{z=\mathrm{e}^{\mathrm{i}\theta}}\mathrm{d}\theta,\\
B^M_b &= \frac{1}{2\pi} \int_0^{2\pi} |z|^2 \left\{2\sigma_1^2(z) - |\sigma_1(z)|^2 + g(b) \right\}|\hat{u}(b,z)|^2\Big|_{z=\mathrm{e}^{\mathrm{i}\theta}}\mathrm{d}\theta.
\end{split}
\end{equation}
Notice that $B^M_a$ and $B^M_b$ are real values, therefore the imaginary parts of the right-hand sides in \eqref{eq60} must integrate to $0$. Therefore,
\begin{equation}
\label{eq61}
\begin{split}
B^M_a &= \!\begin{multlined}[t]
\frac{1}{2\pi} \int_0^{2\pi} |z|^2\left\{ -\left(2\cdot\mathrm{Re}\,\lambda_1^2(z) - g(a)\right)|\hat{u}(a,z)|^2-\lvert\partial_x\hat{u}(a,z)\rvert^2 \right\}\Big|_{z=\mathrm{e}^{\mathrm{i}\theta}}\mathrm{d}\theta \\
+\frac{1}{2\pi} \int_0^{2\pi} |z|^2\left\{ -2\cdot\mathrm{Re}\,\left(\lambda_1(z)\partial_x\hat{u}(a,z) \overline{\hat{u}(a,z)}\right)\right\}\Big|_{z=\mathrm{e}^{\mathrm{i}\theta}}\mathrm{d}\theta,
\end{multlined} \\
B^M_b &= \frac{1}{2\pi} \int_0^{2\pi} |z|^2 \left(2\cdot\mathrm{Re}\,\sigma_1^2(z) - |\sigma_1(z)|^2 + g(b) \right)|\hat{u}(b,z)|^2\Big|_{z=\mathrm{e}^{\mathrm{i}\theta}}\mathrm{d}\theta.
\end{split}
\end{equation}
The quantities $B^M_a$ and $B^M_b$ are now in the same form as \cite[Sect. 2.2]{fang18}, therefore the result follows by~\cite[Prop. 2.4]{besse16a}.
\end{proof}

\section{Spatial discretization: pseudo-spectral approach}
\label{spacedisc}
The spatial discretization of problem~\eqref{eq14a}--\eqref{eq14g} is carried out by a dual Petrov--Galerkin method. In particular, we follow the approach given in~\cite{shen04} for the dispersive part and the approach given in~\cite{shen07} for the variable coefficient advection. It is very well known that pseudo-spectral methods achieve high accuracy even for a modest number of collocation points $N$, provided the solution is smooth. However, these methods have to be carefully designed in order to obtain sparse mass and stiffness matrices in frequency space. Then, the associated linear system can be solved in $\mathcal{O}(N)$ operations.

In the following description we assume, without loss of generality, $a=-1$ and $b=1$. The idea is to choose the dual basis functions of the dual Petrov--Galerkin formulation in such a manner that boundary terms from integration by parts vanish. Let us introduce a variational formulation for
\begin{align}
u^* &= \left(I-\frac{\tau}{2}p_g\partial_x-\frac{\tau}{2}\partial_x^3\right)u^m, \label{eq15} \\
\left(I+\frac{\tau}{2}g^*\partial_x\right)u^{m+1/2} & = \left(I-\frac{\tau}{2}g^*\partial_x\right)u^*,\label{eq16}\\
\left(I+\frac{\tau}{2}p_g\partial_x + \frac{\tau}{2}\partial_x^3\right)u^{m+1} &= u^{m+1/2}\label{eq17}
\end{align}
so that the discrete transparent boundary conditions are satisfied. To this goal, let $\mathcal{P}_N$ be the space of polynomials up to degree $N$. For \eqref{eq15} and \eqref{eq17} we introduce the \emph{dispersive} space
\[
\begin{split}
V^d_N = \{\phi^d\in\mathcal{P}_N| \partial_x^2 \phi^d(a)+Y_1^0\partial_x \phi^d(a)+\left(g_a+Y_2^0\right)\phi^d(a) &= 0,\\
\partial_x \phi^d(b)-Y_3^0\phi^d(b) &= 0,\\
\partial_x^2 \phi^d(b)-Y_4^0\phi^d(b) &= 0\}.
\end{split}
\]
The conditions in $V^d_N$ collect the left-hand side of \eqref{eq14e}--\eqref{eq14g}. Let $(u,v) = \int_{a}^b u(x)v(x)\,\mathrm{d}x$ be the usual $L_2$ inner product. The dual space $V_N^{d,*}$ is defined in the usual way, i.e. for every $\phi^d\in V_N^d$ and $\psi^d\in V_N^{d,*}$ it holds \[
(p_g\partial_x \phi^d  + \partial_x^3 \phi^d,\psi^d) = - (\phi^d, \partial_x (p_g\psi^d) + \partial_x^3 \psi^d).
\]
\begin{lemma}
The dual space $V_N^{d,*}$ of $V_N^d$ is given by
\[
\begin{split}
V^{d,*}_N = \{\psi^d\in\mathcal{P}_N|  \partial_x^2 \psi^d(b)-Y_3^0\partial_x \psi^d(b)+\left(g_b+Y_4^0\right)\psi^d(b) &= 0,\\
\partial_x \psi^d(a)+Y_1^0\psi^d(a) &= 0,\\
\partial_x^2 \psi^d(a) - Y_2^0\psi^d(a) &= 0\}.
\end{split}
\]
\end{lemma}
\begin{proof}
Integrating $(p_g\partial_x \phi^d,\psi^d)$ by parts and integrating $(\partial_x^3 \phi^d,\psi^d)$ by parts three times gives
\[
\begin{split}
(p_g\partial_x \phi^d  + \partial_x^3 \phi^d,\psi^d) & = \int_{a}^{b} \left(p_g\partial_x \phi^d(x)  + \partial_x^3 \phi^d(x) \right) \psi^d(x)\, \mathrm{d}x  \\
&= \!\begin{multlined}[t] p_g\cdot\phi^d \cdot \psi^d\rvert_{x=a}^b + \partial_x^2 \phi^d\cdot \psi^d\rvert_{x=a}^b - \partial_x \phi^d \cdot \partial_x \psi^d \rvert_{x=a}^b \\+ \phi^d \cdot\partial_x^2 \psi^d \rvert_{x=a}^b - (\phi^d, \partial_x (p_g\cdot\psi^d) + \partial_x^3 \psi^d). \end{multlined}
\end{split}
\]
We want the boundary terms to vanish. For $x=b$ we have
\[
\begin{split}
0 &= p_g(b) \phi^d(b) \psi^d(b) + \partial_x^2 \phi^d(b)\psi^d(b) -\partial_x\phi^d(b)\partial_x\psi^d(b) + \phi^d(b)\partial_x^2\psi^d(b) \\
&= \phi^d(b)\cdot \left(\partial_x^2 \psi^d(b) -Y_3^0\partial_x\psi^d(b) + (g_b+Y_4^0)\psi^d(b)\right).
\end{split}
\]
The last equality is obtained by substituting $p_g(b) = g_b$ and using the relations given by the space $V_N^d$ for $\partial_x\phi^d(b)$ and $\partial_x^2\phi^d(b)$. Similarly for $x=a$ we have
\[
\begin{split}
0 &= p_g(a) \phi^d(a) \psi^d(a) + \partial_x^2 \phi^d(a)\psi^d(a) -\partial_x\phi^d(a)\partial_x\psi^d(a) + \phi^d(a)\partial_x^2\psi^d(a)\\
&= \phi^d(a)\cdot \left(\partial_x^2\psi^d(a) - Y_2^0\psi^d(a)\right) + \partial_x\phi^d(a)\cdot\left(\partial_x\psi^d(a) + Y_1^0\psi^d(a)\right),
\end{split}
\]
which leads to the boundary relations of the dual space $V_N^*$.
\end{proof}
We proceed by introducing the \emph{advection} space for \eqref{eq16}:
\[
V^a_N = \{\phi^a\in\mathcal{P}_N\}.
\]
Notice that due to the variable coefficient $g^*$ the space $V^a_N$ is free from inflow or outflow conditions. The dual space $V_N^{a,*}$ is defined so that for every $\phi^a\in V_N^a$ it holds
\[
(g^*\partial_x \phi^a,\psi^a) = - \left(\phi^a,\partial_x (g^* \psi^a)\right)
\] for every $\psi^a\in V_N^{a,*}$.
The dual space $V^{a,*}_N$ is $\mathcal{P}_N$.
Let $L_j$ be the $j$th Legendre polynomial. We define
\begin{equation}
\label{eq18}
\begin{split}
\phi^d_j(x) &:= L_j(x) + \alpha_jL_{j+1}(x) + \beta_j L_{j+2}(x) + \gamma_j L_{j+3}(x),\quad 0\leq j\leq N-3,\\
\psi^d_j(x) &:= L_j(x) + \alpha_j^*L_{j+1}(x) + \beta_j^* L_{j+2}(x) + \gamma_j^* L_{j+3}(x),\quad 0\leq j\leq N-3,\\
\phi^a_j(x) &:= L_j(x),\quad 0\leq j\leq N, \\
\psi^a_j(x) &:= L_j(x),\quad 0\leq j\leq N, \\
\end{split}
\end{equation}
where the coefficients $\alpha_j,\beta_j,\gamma_j,\alpha_j^*,\beta_j^*,\gamma_j^*$ are chosen in such a way that $\phi_j^d$, $\psi_j^d$ belong to $V_N^d$, $V_N^{d,*}$, respectively, see appendix~\ref{app1}.
The sequences $\{\phi_j^d\}_{j=0}^{N-3}$ and $\{\psi_j^d\}_{j=0}^{N-3}$ are a basis of $V_N^d$ and $V_N^{d,*}$ respectively. We are now ready to consider the variational formulation.
\subsection{Variational formulation}
The dual Petrov--Galerkin formulation of~\eqref{eq15} reads: find $u^*\in\mathcal{P}_N$ such that
\begin{equation}
\label{eq19}
(u^*,\psi^d_j) = \left(\left(I-\frac{\tau}{2}p_g\partial_x-\frac{\tau}{2}\partial_x^3\right)u^m,\psi^d_j\right)
\end{equation}
holds for every $\psi^d_j\in V_N^{d,*}$, $j=0,\dots ,N-3$.
In general the function $u^m$ does not belong to the space $V_N^d$. Indeed $u^m$ satisfies the discrete transparent boundary conditions
\[
\begin{split}
 \partial_x^2 u^{m}(a)  + Y_1^0\partial_x u(a)^{m} +\left(g_a +  Y_2^0\right) u(a)^{m} &= h_1^m,\\
 \partial_x u^{m}(b) - Y_3^0u(b)^{m} &= h_2^m,\\
\partial_x^2 u^{m}(b)  - Y_4^0 u(b)^{m} &= h_3^m.
\end{split}
\]
However, we can write $u^m = u^m_h + p_2^m$, where $u^m_h\in V_N^d$ and $p_2^m$ is the unique polynomial of degree two such that
\begin{equation}
\label{eq29}
\begin{split}
 \partial_x^2 p_2^{m}(a)  + Y_1^0 \partial_xp_2^{m}(a) + \left(g_a + Y_2^0\right) p_2^{m}(a) &= h_1^m,\\
 \partial_x p_2^{m}(b) - Y_3^0 p_2^{m}(b) &= h_2^m,\\
 \partial_x^2 p_2^{m}(b)  - Y_4^0 p_2^{m}(b) &= h_3^m.
\end{split}
\end{equation}
The function $u^*$ also does not belong to the space. Similarly, we can write $u^* = u^*_h + p_2^*$. We assume that $u^*$ satisfies the same boundary conditions as $u^m$, therefore $p^*_2 = p^m_2$. We thus obtain 
\begin{equation}
\label{eq20}
(u^*_h,\psi^d_j) = \left(\left(I-\frac{\tau}{2}p_g\partial_x -\frac{\tau}{2}\partial_x^3\right)u^m_h,\psi^d_j\right) +\Big(\underbrace{p_2^m-p_2^*}_{=0} -\frac{\tau}{2}p_g\partial_x p_2^m,\psi_j^d\Big).
\end{equation}
We proceed with the dual Petrov--Galerkin formulation of~\eqref{eq16}. Find $u^{m+1/2}\in\mathcal{P}_N$ such that
\begin{equation}
\label{eq21}
\left(\left(I+\frac{\tau}{2}g^*\partial_x\right) u^{m+1/2},\psi^a_j\right) = \left(\left(I-\frac{\tau}{2}g^*\partial_x\right) u^*,\psi_j^a\right)
\end{equation}
holds for every $\psi^a_j\in V_N^{a,*}$, $j=0,\dots ,N$.  Notice that $u^*, u^{m+1/2}\in V_N^a$.
So, differently from the dispersive case, we obtain the solution $u^{m+1/2}$ without performing any shift. 

Similarly to~\eqref{eq15}, the dual Petrov--Galerkin formulation of~\eqref{eq17} reads: find $u^{m+1}\in\mathcal{P}_N$ such that
\begin{equation}
\label{eq25}
\left(\left(I+\frac{\tau}{2}p_g\partial_x+\frac{\tau}{2}\partial_x^3\right)u_h^{m+1},\psi^d_j\right) = \left(u^{m+1/2} - p_2^{m+1} -\frac{\tau}{2}p_g\partial_x p^{m+1}_2,\psi_j^d\right)
\end{equation}
holds for every $\psi^d_j\in V_N^{d,*}$, $j=0,\dots ,N-3$.

\subsection{Implementation in frequency space}
\label{freqspace}
This section is dedicated to compute the mass and stiffness matrices for~\eqref{eq20}--\eqref{eq25}. When it comes to numerical implementation, the $L^2$ inner product $(u,v)$ needs to be approximated. We use two different discrete inner products for the spaces $V^d_N$ and $V_N^a$. This choice is motivated by the fact that the spaces $V_N^d$ and $V_N^a$ satisfy different boundary conditions.
\begin{definition}[Dispersive inner product]
Let $\langle \cdot,\cdot\rangle_N^d$ be the dispersive inner product defined as
\begin{equation}
\label{eq27}
\langle u,v\rangle_N^d := \sum_{\ell = 2}^{N-1}w_{\ell}u(y_{\ell})v(y_{\ell}) + w_1u(-1)v(-1) + w_{N}u(1)v(1) + w'_{N}\partial_y\left(u(y)v(y)\right)\bigg|_{y=1},
\end{equation}
where $y_l$ are the roots of the Jacobi polynomial $P^{(2,1)}_{N-2}(y)$ and $w_l$ the associated weights.
\end{definition}

\begin{definition}[Advection inner product]
Let $\langle \cdot,\cdot\rangle_N^a$ be the advection inner product defined as
\begin{equation}
\label{eq35}
\langle u,v\rangle_N^a := \sum_{\ell = 2}^{N+2}w_{\ell}u(y_{\ell})v(y_{\ell})
\end{equation}
where $y_l$ are the roots of the Jacobi polynomial $P^{(0,0)}_{N+1}(y)$ and $w_l$ the associated weights.
\end{definition}
We have $(u,v) = \langle u,v\rangle_N^d$ for all polynomials $u$, $v$ such that $\deg u + \deg  v \leq 2N-2$ and $(u,v) = \langle u,v\rangle_N^a$ for all polynomials $u$, $v$ such that $\deg u + \deg  v \leq 2N+1$. For more details about generalized quadrature rules, we refer the reader to~\cite{huang92}.
\bigskip

\noindent\emph{Stiffness and mass matrices for~\eqref{eq20}, \eqref{eq25}}.
Since $u^m_h\in V_N^d$, we can express it as linear combination of $V_N^d$ basis functions, i.e.
\begin{equation}
\label{eq20b}
u^m_h(x) = \sum_{k=0}^{N-3} \tilde{u}_{h,k}^{m,d}\phi^d_k(x).
\end{equation}
The first step is to obtain the frequency coefficients $\tilde{u}_{h,k}^{m,d}$ in \eqref{eq20b}. We take the dispersive inner product on both sides
\begin{equation}
\label{eq26}
\langle u^m_h,\psi_j^d\rangle^d_N = \sum_{k=0}^{N-3} \tilde{u}_{h,k}^{m,d} \langle\phi^d_k,\psi^d_j\rangle^d_N.
\end{equation}
The mass matrix is
\[
\mathbf{M}^d\in \mathbb{R}^{(N-2)\times(N-2)},\quad\mathbf{M}^d_{kj} := \langle\phi^d_k,\psi^d_j\rangle_N^d.
\]
Using the orthogonality relation between $\phi^d_k$ and $\psi^d_j$ gives $\langle\phi^d_k,\psi^d_j\rangle^d_N = 0$ if $|k-j|>3$ and $j+k\leq 2N-8$, see appendix~\ref{app2}. Then, $\mathbf{M}^d$ is a $7$-diagonal matrix. Equation~\eqref{eq26} in matrix form reads
\begin{equation}
\label{eq27b}
\langle u^m_h,\psi_j^d\rangle_N^d = [(\mathbf{M}^d)^T\tilde{\mathbf{u}}_{h}^{m,d}]_j.
\end{equation}
The left-hand side of~\eqref{eq27b} can also be written in matrix form:
\[
\begin{split}
\langle u^m_h,\psi_j^d\rangle_N^d =  \sum_{\ell = 2}^{N-1} & w_{\ell}u^m_h(y_{\ell})\psi_j^d(y_{\ell}) \\
& + \underbrace{w_1u^m_h(-1)\psi^d_j(-1) + w_{N}u^m_h(1)\psi^d_j(1) + w'_{N}\partial_y\left(u^m_h(y)\psi^d_j(y)\right)\bigg|_{y=1}}_{\mathbf{b}_j} \\
& = [\Psi^{d^T}\Omega\,\mathbf{u}^m_h]_j + \mathbf{b}_j,
\end{split}
\]
where
\[
\Psi^d = \begin{bmatrix}
\psi^d_0(y_2) & \dots & \psi^d_{N-3}(y_2)\\
\psi^d_0(y_3) & \dots &\psi^d_{N-3}(y_3)\\
\vdots & & \vdots \\
\psi^d_0(y_{N-1}) & \dots & \psi^d_{N-3}(y_{N-1})
\end{bmatrix},\quad
\Omega = \text{diag}\begin{bmatrix}
w_2\\
\vdots\\
w_{N-1}
\end{bmatrix},\quad
\mathbf{u}^m_h = \begin{bmatrix}
u^m_h(y_2)\\
\vdots\\
u^m_h(y_{N-1})\\
\end{bmatrix}.
\]
We obtain the frequency coefficients
\[
\tilde{\mathbf{u}}_h^{d,m} = (\mathbf{M}^{d})^{-T}\left( \Psi^{d^T}\Omega\,\mathbf{u}^m_h+\mathbf{b}\right).
\]
The second step is to compute the stiffness matrix and the frequency coefficients of the second term of the addition in~\eqref{eq20}. The stiffness matrix is
\[
\mathbf{S}^d\in\mathbb{R}^{(N-2)\times(N-2)},\quad
\mathbf{S}^d_{kj} =\langle p_g\partial_x \phi_k^d + \partial^3_x\phi^d_k,\psi^d_j\rangle_N^d.
\]
\begin{lemma}
$\mathbf{S}^d$ is a 7-diagonal matrix.
\end{lemma}
\begin{proof}
We now that $\phi^d_k$ is a polynomial of degree $k+3$. Therefore, $ q(x):= p_g(x)\partial_x \phi_k^d(x) + \partial^3_x\phi^d_k(x)$ is a polynomial of degree $\leq k+3$. We write $q$ as a linear combination of Legendre polynomials up to degree $k+3$:
\[
q(x) = \sum_{i=0}^{k+3} q_i L_i(x).
\]
Let us consider the dispersive inner product $\langle q,\psi^d_j\rangle^d_N$ and $k+3<j$. Then,
\[
\langle q,\psi^d_j \rangle^d_N = \sum_{i=0}^{k+3}q_i \langle L_i,\psi^d_j \rangle^d_N = \sum_{i=0}^{k+3}q_i \langle L_i,L_j + \alpha^*_j L_{j+1} + \beta^*_jL_{j+2} + \gamma^*_j L_{j+3}\rangle^d_N = 0.
\]
The last equation follows from the definition of $\psi^d_j$ and the orthogonality property of the Legendre polynomials. Let $j < k+3$ with $k+j\leq 2N-8$, then (see appendix~\ref{app2})
\[
\langle q,\psi^d_j\rangle^d_N = \langle p_g\partial_x \phi_k^d + \partial^3_x\phi^d_k, \psi^d_j\rangle^d_N = -\langle \phi_k^d, \partial_x(p_g\psi_j^d)+\partial_x^3\psi_j^d\rangle^d_N.
\]

The polynomial $\tilde{q} = \partial_x(p_g\psi_j^d)+\partial_x^3\psi_j^d$ is of degree $j$.  Similarly to $q$, we obtain $\langle \phi_k^d,\tilde{q}\rangle^d_N=0$ and the result follows.
\end{proof}
The frequency coefficients  of the second term on the right-hand side of~\eqref{eq20} are given by
\[
\tilde{\mathbf{p}}^m\in\R^{N-2},\quad\tilde{\mathbf{p}}^m_j = \langle p_g\partial_x p_2^m,\psi_j^d\rangle_N^d,\quad j=0,\dots,N-3.
\]
Notice that $p_g\partial_xp_2^m$ is a polynomial of degree 2. Therefore, it can be written as a linear combination of the Legendre polynomials $L_0$, $L_1$ and $L_2$. Using the orthogonality property of Legendre polynomials we obtain $\langle p_g\partial_xp_2^m,\psi_j^d\rangle_N^d=0$ for $j>2$.
Problem \eqref{eq20} is equivalent to
\begin{equation}
\label{eq28}
(\mathbf{M}^d)^{T} \tilde{\mathbf{u}}^{*,d}_h = \left(\mathbf{M}^d-\frac{\tau}{2}  \mathbf{S}^d\right)^T\tilde{\mathbf{u}}^{m,d}_h -\frac{\tau}{2}\tilde{\mathbf{p}}^m.
\end{equation}

A similar procedure applies to \eqref{eq25}, where we obtain
\begin{equation}
\label{eq28b}
\left(\mathbf{M}^d + \frac{\tau}{2}\mathbf{S}^d\right)^T\tilde{\mathbf{u}}^{m+1,d}_h = \tilde{\mathbf{u}}^{m+1/2,d}-\tilde{\mathbf{p}}_2^{m+1} -\frac{\tau}{2}\tilde{\mathbf{p}}^{m+1}
\end{equation}
with
\[
\tilde{\mathbf{p}}^{m+1}_{2}\in\R^{N-2},\quad \tilde{\mathbf{p}}^{m+1}_{2,j} = \langle p^{m+1}_2,\psi^d_j\rangle^d_N\quad \text{for } j = 0,\dots,N-3.
\]
Both linear systems \eqref{eq28}-\eqref{eq28b} can be solved in $\mathcal{O}(N)$ operations since $\mathbf{M}^d$ and $\mathbf{S}^d$ are 7-diagonal matrices.
\bigskip

\noindent\emph{Stiffness and mass matrices for~\eqref{eq21}}. We can express the functions $u^*$ and $u^{m+1/2}$ as linear combinations of $V_N^a$ basis functions, i.e.
\begin{align}
\label{eq20c}
u^*(x) &= \sum_{k=0}^{N} \tilde{u}_{k}^{*,a}\phi^a_k(x),\\
\label{eq20d}
u^{m+1/2}(x) &= \sum_{k=0}^{N} \tilde{u}_{k}^{m+1/2,a}\phi^a_k(x).
\end{align}
Similarly to the dispersive case, we need the frequency coefficients in~\eqref{eq20c}, \eqref{eq20d}. We take the advection inner product in~\eqref{eq20c}, \eqref{eq20d} on both sides
\begin{align}
\label{eq34}
\langle u^*,\psi_j^a\rangle^a_N &= \sum_{k=0}^{N} \tilde{u}_{k}^{*,a}\langle \phi^a_k,\psi^a_j\rangle^a_N,\\
\label{eq34b}
\langle u^{m+1/2},\psi^a_j\rangle_N^a &= \sum_{k=0}^{N} \tilde{u}_{k}^{m+1/2,a}\langle \phi^a_k,\psi^a_j\rangle_N^a.
\end{align}
Using the orthogonality relation between $\phi^a_k$ and $\psi^a_j$ gives $\langle \phi^a_k,\psi^a_j\rangle^a_N = 0$ if $k\neq j$. Then, the mass matrix
\[
\mathbf{M}^a\in\R^{(N+1)\times (N+1)},\quad \mathbf{M}^a_{kj} =  \langle\phi_k^a,\psi_j^a\rangle_N^a
\]
is a diagonal matrix.
Finally, Problem~\eqref{eq21} is equivalent to
\begin{equation}
\label{eq37}
\left(\mathbf{M}^a + \frac{\tau}{2}\mathbf{S}^a\right)^T\tilde{\mathbf{u}}^{m+1/2,a} = \left(\mathbf{M}^a - \frac{\tau}{2}\mathbf{S}^a\right)^T\tilde{\mathbf{u}}^{*,a}
\end{equation}
with the stiffness matrix $\mathbf{S}^a\in\R^{(N+1)\times (N+1)}$ defined by
\begin{equation}
\label{eq37b}
\mathbf{S}^a_{kj} = \langle g^*\partial_x \phi_k^a,\psi_j^a\rangle_N^a.
\end{equation}
The stiffness matrix $\mathbf{S}^a$ is in general a full matrix. A direct inversion of~\eqref{eq37} requires $\mathcal{O} (N^3)$ operations, thus is not advisable. Applying an iterative scheme is preferable, but multiplying the matrix $\mathbf{S}^a$ with a vector costs $\mathcal O(N^2)$ operations. A more efficient way is to compute $g^*\partial_x u^*$ (and $g^*\partial_x u^{m+1/2}$) in the physical space at the advection collocation points. The point-wise multiplication with $g^*$ costs only $\mathcal O(N)$ operations. The result is then transformed back to the frequency space. Transforming back and forth to the frequency space can be done efficiently by employing the discrete Lagrange transform (DLT) and the inverse discrete Lagrange transform (IDLT) developed in~\cite{hale15}, see appendix~\ref{app3}.

\begin{remark}
For the special case where $g^*$ is a polynomial of degree $n$, the stiffness matrix $\mathbf{S}^a$ is banded with  bandwidth less or equal to $2n$. This implies that for a small $n$ the linear system~\eqref{eq37} is sparse and can be solved in $\mathcal O(N)$ operations without switching from the frequency to the physical space.
\end{remark}
\bigskip

\noindent\emph{Transition matrices}. In order to connect~\eqref{eq37} to \eqref{eq28} and \eqref{eq28b}, it is necessary to transfer information from the dispersive space to the advection space and vice-versa. In particular, the aim is to translate the frequency coefficients from the dispersive space to the advection space in an efficient way.  Let
\[
\begin{split}
& \mathbf{M}^{da}\in\R^{(N-2)\times (N+1)},\quad \mathbf{M}^{da}_{kj} = \langle\phi_k^d,\psi_j^a\rangle_N^a,\\
& \mathbf{M}^{ad}\in\R^{(N+1)\times (N-2)},\quad \mathbf{M}^{ad}_{kj} = \langle\phi_k^a,\psi_j^d\rangle_N^d.
\end{split}
\]
By using the orthogonality property of Legendre polynomials, one can prove that $\mathbf{M}^{da}$ and $\mathbf{M}^{ad}$ are 4-diagonal matrices.
Consider
\[
\langle u^*,\psi^a_j\rangle_N^a = \sum_{k=0}^{N} \tilde{u}_{k}^{*,a}\langle\phi^a_k,\psi^a_j\rangle_N^a\quad \text{and}\quad \langle u^*,\psi^a_j\rangle_N^a= \sum_{k=0}^{N-3} \tilde{u}_{k}^{*,d}\langle \phi^d_k,\psi^a_j\rangle_N^a,
\]
for $j=0,\dots,N.$ Then,
\[
(\mathbf{M}^a)^T\tilde{\mathbf{u}}^{*,a} = (\mathbf{M}^{da})^T\tilde{\mathbf{u}}^{*,d}.
\]
The frequency coefficients $\tilde{\mathbf{u}}^{*,a}$ are obtained directly from the coefficients $\tilde{\mathbf{u}}^{*,d}$ in $\mathcal{O}(N)$ operations. Similarly, consider
\[
\langle u^{m+1/2,d},\psi^d_j\rangle_N^d = \sum_{k=0}^{N-3} \tilde{u}_{k}^{m+1/2,d}\langle\phi^d_k,\psi^d_j\rangle_N^d\quad \text{and}\quad \langle u^{m+1/2},\psi^d_j\rangle_N^d= \sum_{k=0}^{N-1} \tilde{u}_{k}^{m+1/2,a}\langle \phi^a_k,\psi^d_j\rangle_N^d,
\]
for $j=0,\dots, N-3$. Then
\[
(\mathbf{M}^d)^T\tilde{\mathbf{u}}^{m+1/2,d} = (\mathbf{M}^{ad})^T\tilde{\mathbf{u}}^{m+1/2,a}.
\]
The coefficients $\tilde{\mathbf{u}}^{m+1/2,d}$ can be directly obtained from $\tilde{\mathbf{u}}^{m+1/2,a}$ in $\mathcal{O}(N)$ operations.
\bigskip

\noindent\emph{Full discretization}. The implementation in frequency space results in
 \begin{align}
\label{eq36}
(\mathbf{M}^d)^T \tilde{\mathbf{u}}^{*,d}_h &= \left(\mathbf{M}^d-\frac{\tau}{2}  \mathbf{S}^d\right)^T\tilde{\mathbf{u}}^{m,d}_h -\frac{\tau}{2}\tilde{\mathbf{p}}^m, \\
(\mathbf{M}^a)^T\tilde{\mathbf{u}}^{*,a} &= (\mathbf{M}^{da})^T(\tilde{\mathbf{u}}^{*,d}_h + \tilde{\mathbf{p}}_2^{m}),\\
\left(\mathbf{M}^a + \frac{\tau}{2}\mathbf{S}^a\right)^T\tilde{\mathbf{u}}^{m+1/2,a} &= \left(\mathbf{M}^a - \frac{\tau}{2}\mathbf{S}^a\right)^T\tilde{\mathbf{u}}^{*,a}\\
(\mathbf{M}^d)^T\tilde{\mathbf{u}}^{m+1/2,d} &= (\mathbf{M}^{ad})^T\tilde{\mathbf{u}}^{m+1/2,a},\\
\left(\mathbf{M}^d + \frac{\tau}{2}\mathbf{S}^d\right)^T\tilde{\mathbf{u}}^{m+1,d}_h &= \tilde{\mathbf{u}}^{m+1/2,d}-\tilde{\mathbf{p}}_2^{m+1} -\frac{\tau}{2}\tilde{\mathbf{p}}^{m+1}.
\end{align}
The solution $u^{m+1}(x)$ can be reconstructed by
\begin{equation}
\label{eq29b}
u^{m+1}(x) = \sum_{k=0}^{N-3}\tilde{\mathbf{u}}_{h,k}^{m+1}\phi^d_k(x) + p^{m+1}_2(x).
\end{equation}


\section{Numerical results}
\label{numres}
In this section, we present numerical results that illustrate the theoretical investigations of the previous chapters. For that purpose, we consider
\begin{equation}
\label{eq30}
\begin{cases}
\partial_t u + g \partial_x u + \partial_x^3 u = 0, \quad (t,x)\in [0,T]\times \R,\\
 u(0,x) = u^0(x)\\
\end{cases}
\end{equation}
with final time $T=1$ and initial value $u^0(x)=\mathrm{e}^{-x^2}$.
We restrict~\eqref{eq30} to the interval $(-6,6)$ and impose transparent boundary conditions at $x=\pm 6$.
The initial data is chosen such that $|u^0(\pm 6)|\leq 10^{-15}$.

For the numerical simulations, we employ a time discretization with constant step size
\[
\tau = T/M,\quad t^m = \tau m,\quad 0\leq m\leq M
\] and a space discretization given by the dual-Petrov--Galerkin variational formulation with $N$ collocation points. We consider the error $\ell^2$ of the full discretization defined as
\[
\Vert\text{err}\Vert_{\ell^2} = \sqrt{\tau \sum_{m=1}^M (\text{err}^m)^2},
\]
where
\[
\text{err}^m = \sqrt{ \frac{ \sum_{j} \left(u^m_{\text{ref}}(x_j)-u^m_N(x_j)\right)^2}{\sum_{j} \left(u^m_{\text{ref}}(x_j)\right)^2} }
\]
is the relative $\ell^2$ spatial error computed at time $t^m = \tau m$. The points $x_j$ are chosen to be equidistant in $[-6,6]$ with $0\leq j\leq J=2^7$. Finally, the function
$u^m_{\text{ref}}$ is either a reference solution or the exact solution, if available. The function $u^m_N$ is the numerical solution at time $t^m$ employing $N$ collocation points.

\begin{example}[Constant advection]
We consider~\eqref{eq30} with constant advection $g(x)=6$. This is the same problem which is considered in~\cite{besse16}. The setting reduces the advection equation in the modified splitting~\eqref{eq2b} to the identity map. Even if the time-splitting is trivial, for this particular problem the exact solution can be computed via Fourier transform, see~\cite{besse16}. Consequently the constant advection problem offers a good benchmark for testing the convergence of the proposed numerical method in that context.

\begin{figure}[t]
\centering
\includegraphics[scale=.3]{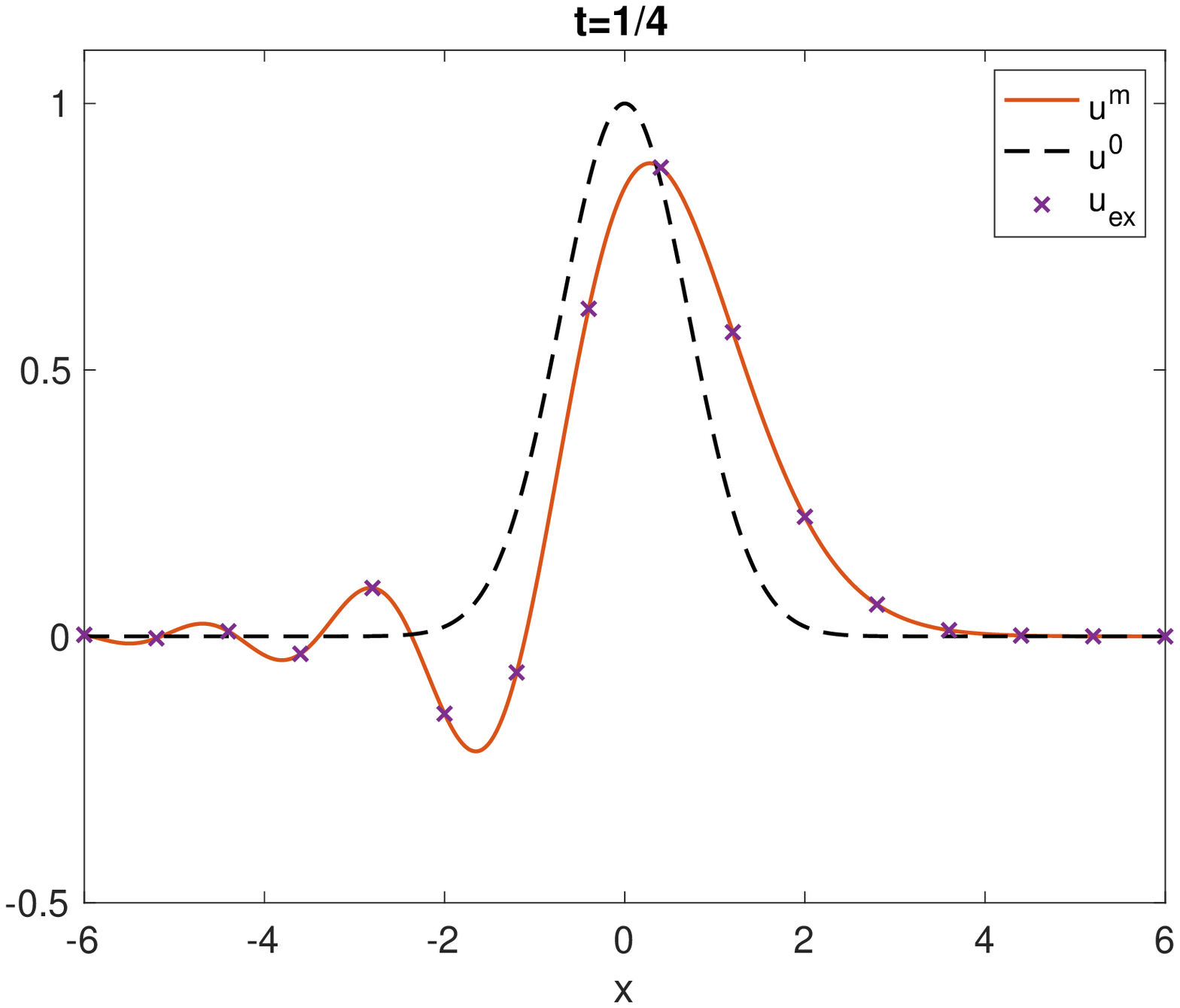}
\includegraphics[scale=.3]{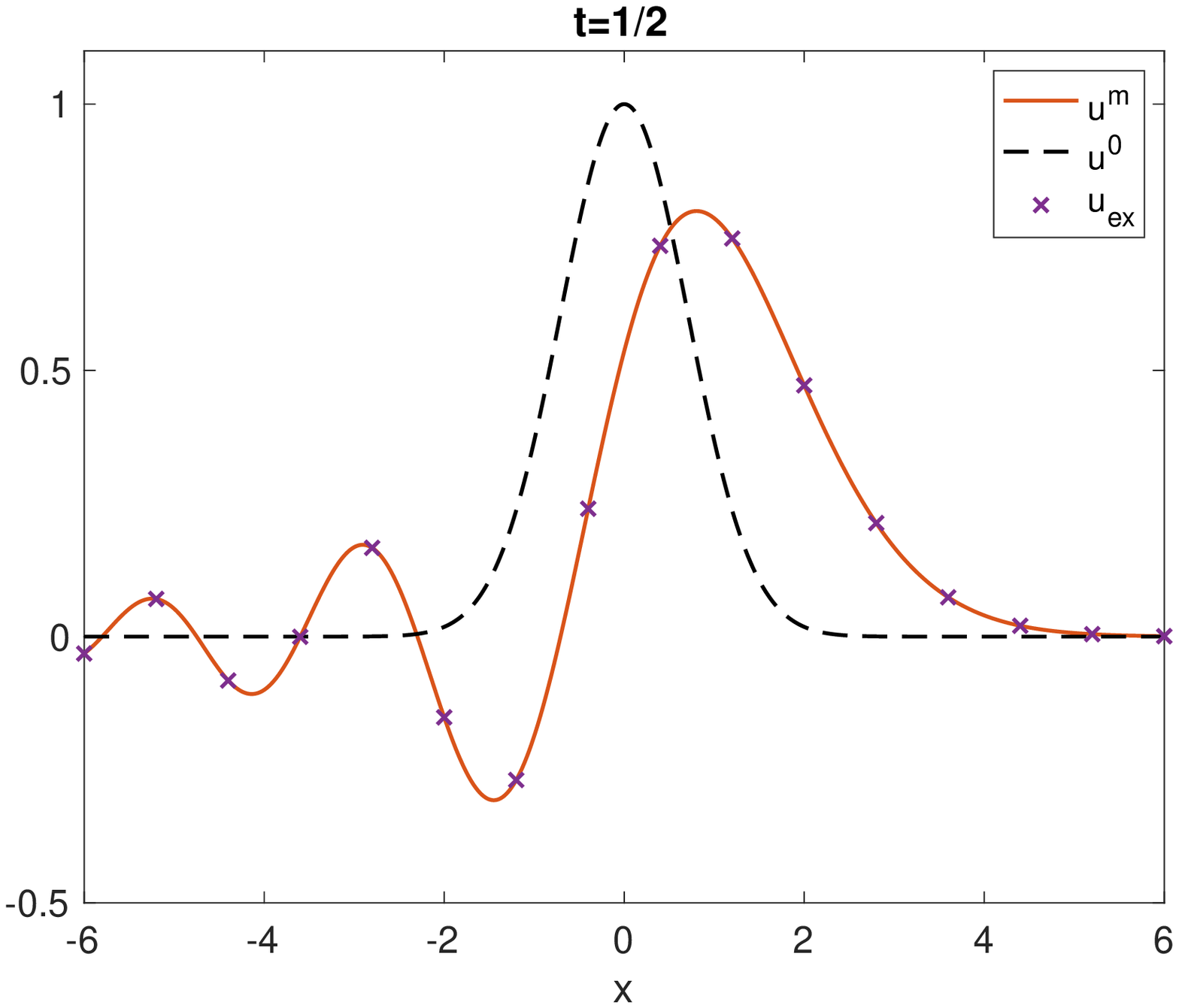}
\includegraphics[scale=.3]{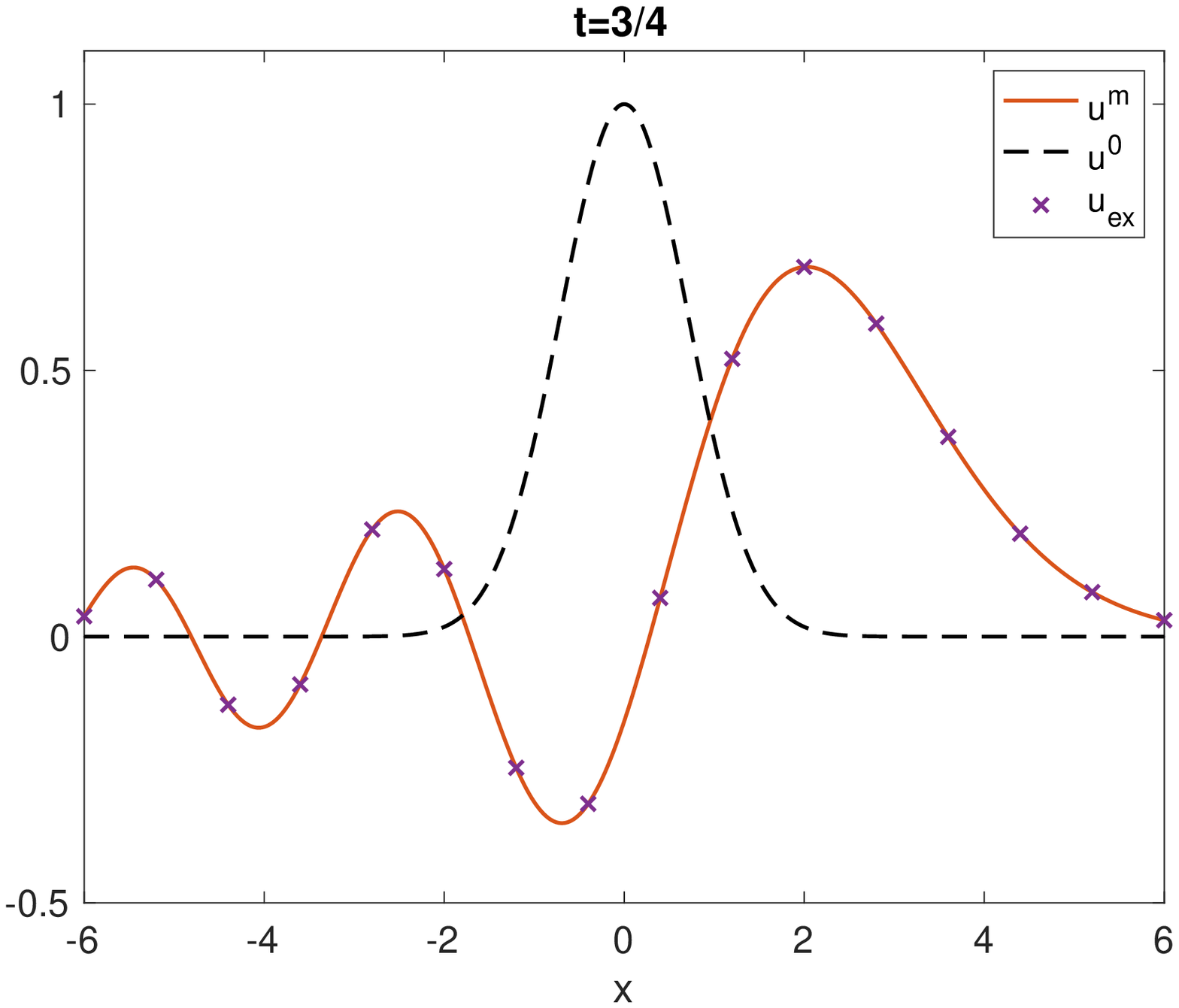}
\includegraphics[scale=.3]{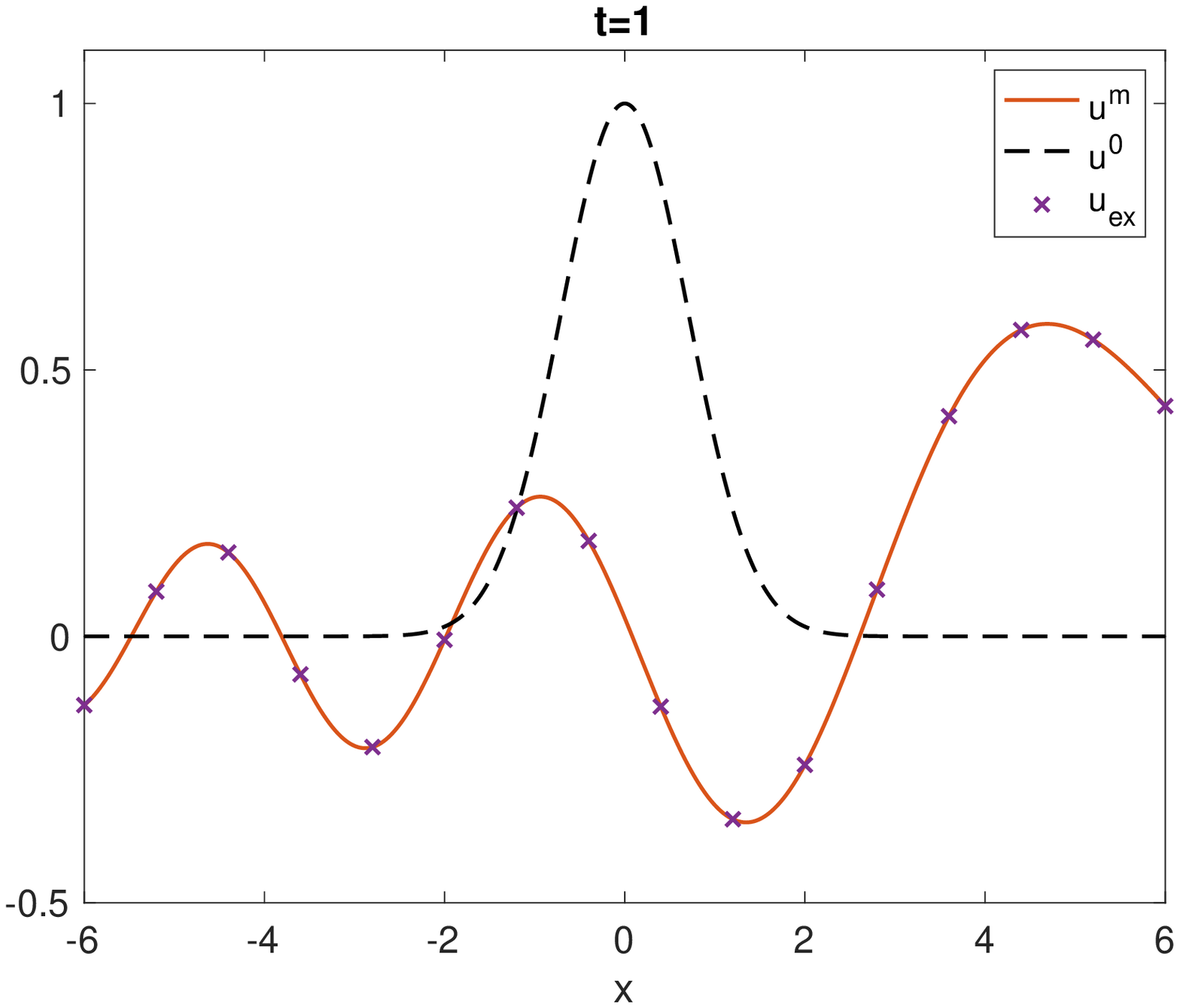}
\caption{Snapshots of the exact solution $u_{\mathrm{ex}}$ and the numerical solution $u^m_N$ for $g=6$ and $t=\frac{1}{4}$, $\frac{1}{2}$, $\frac{3}{4}$, $1$ with $\tau=2^{-12}$. The number of collocation points is set to $N=2^6$. We notice that the cross marks representing the exact solution lie on the numerical solution.}
\label{fig1}
\end{figure}

In Fig.~\ref{fig1} snapshots of the numerical solution $u^m_N$ for $t = \frac{1}{4},$ $\frac{1}{2},$ $\frac{3}{4}$, $1$ and $\tau = 2^{-12}$ are shown. Notice that the numerical solution ``leaves'' the domain at the boundary $x=-6$ without any reflection. As time increases the solution moves to the right and re-enters the computational domain. Finally, the solution matches the boundary at $x=6$ without any reflection.

In Table~\ref{tab1} the full discretization error between the numerical solution and the exact solution varying $N$ and $M$ is reported. In particular, in Table~\ref{tab1} (left) the number of time steps $M$ is fixed to $2^{12}$ and the number of collocation points $N$ is varying from 24 to 40. In this way the time discretization error is small enough to be negligible with respect to the spatial error. The value $\alpha$ denotes the slope of the line obtained by connecting two subsequent error values and varying $N$ in a semi-logarithmic plot . More specifically, let $N_1$ and $N_2$ with $N_1<N_2$ be two subsequent values of $N$ and $\lVert\mathrm{err}_1\rVert_{\ell^2}$, $\lVert\mathrm{err}_2\rVert_{\ell^2}$  the associated error values. Then
\[
\frac{\lVert\mathrm{err}_2\rVert_{\ell^2}}{\lVert\mathrm{err}_1\rVert_{\ell^2}} = \exp\left(-\alpha\cdot(N^2_2-N^2_1)\right).
\]
Notice that $\alpha$ remains constant when $N$ is varying, which confirms the spectral accuracy of the numerical scheme.
\bigskip

In Table~\ref{tab1} (right) the number of collocation points is fixed to $2^6$ and the number of time steps $M$ is varying from $2^5$ to $2^8$. In this way the space error is small enough to be negligible with respect to the time error. The value $\beta$ denotes the slope of the line obtained connecting two subsequent error values and varying $M$ in a double-logarithmic plot. More specifically, let $M_1$ and $M_2$ with $M_1<M_2$ be two subsequent values of  $M$ and $\lVert\mathrm{err}_1\rVert_{\ell^2}$, $\lVert\mathrm{err}_2\rVert_{\ell^2}$  the associated error values. Then
\[
\frac{\lVert\mathrm{err}_2\rVert_{\ell^2}}{\lVert\mathrm{err}_1\rVert_{\ell^2}}= \left(\frac{M_2}{M_1}\right)^{-\beta}.
\]
We can clearly see $\beta\approx 2$, which confirms second order accuracy in time.

\begin{table}[!h]
\begin{center}
\small
\begin{tabular}{c c c}
$N$ & $\lVert\mathrm{err}\rVert_{\ell^2}$ & $\alpha$\Bstrut\\
\hline
$24$ & $2.6141\mathrm{e}-03$ & -- \Tstrut\\
$32$ & $8.7517\mathrm{e}-05$ & $7.2821\mathrm{e}-03$ \\
$40$ & $1.8603\mathrm{e}-06$ & $6.8540\mathrm{e}-03$ \\
$48$ & $3.5613\mathrm{e}-08$ & $6.5603\mathrm{e}-03$ \\
\hline
\end{tabular}
\hspace{2cm}
\begin{tabular}{c c c}
$M$ & $\lVert\mathrm{err}\rVert_{\ell^2}$ & $\beta$\Bstrut\\
\hline
$2^5$ & $4.1849\mathrm{e}-04$ & -- \Tstrut\\
$2^6$ & $1.0995\mathrm{e}-04$ & $1.9283$ \\
$2^7$ & $2.7559\mathrm{e}-05$ & $1.9963$ \\
$2^8$ & $6.8668\mathrm{e}-06$ & $2.0048$ \\
\hline
\end{tabular}
\caption{We present the full discretization error $\lVert \mathrm{err}\rVert_{\ell^2}$ for constant $g$. On the left side $M$ is fixed to $2^{12}$ so that the time error is negligible w.r.t.~the spatial error. On the right side $N$ is fixed to $2^6$ so that the spatial error is negligible w.r.t.~the time error. In both tables, errors are obtained testing $u^m_N$ against the exact solution computed via Fourier transform as in~\cite{besse16}. The fact that $\alpha$ remains constant confirms the spectral accuracy of the proposed method, while the fact that $\beta\approx 2$ confirms the second order in time.}
\label{tab1}
\end{center}
\end{table}
\end{example}

\begin{example}
We consider~\eqref{eq30} with $g(x) = - x^3/54 + x + 3$. As mentioned in section~\ref{freqspace}, for $g$ being a low degree polynomial, the stiffness matrix $\mathbf{S}^a$ results in a banded matrix. Therefore, the linear system associated to the advection equation can be solved in $O(N)$ operations. The exact solution for this problem is not known, so we test the numerical solution $u^m_N$ against a reference solution $u^m_{\text{ref}}$ computed using a significantly greater number of points (both in time and space).

\begin{figure}[t]
\centering
\includegraphics[scale=.3]{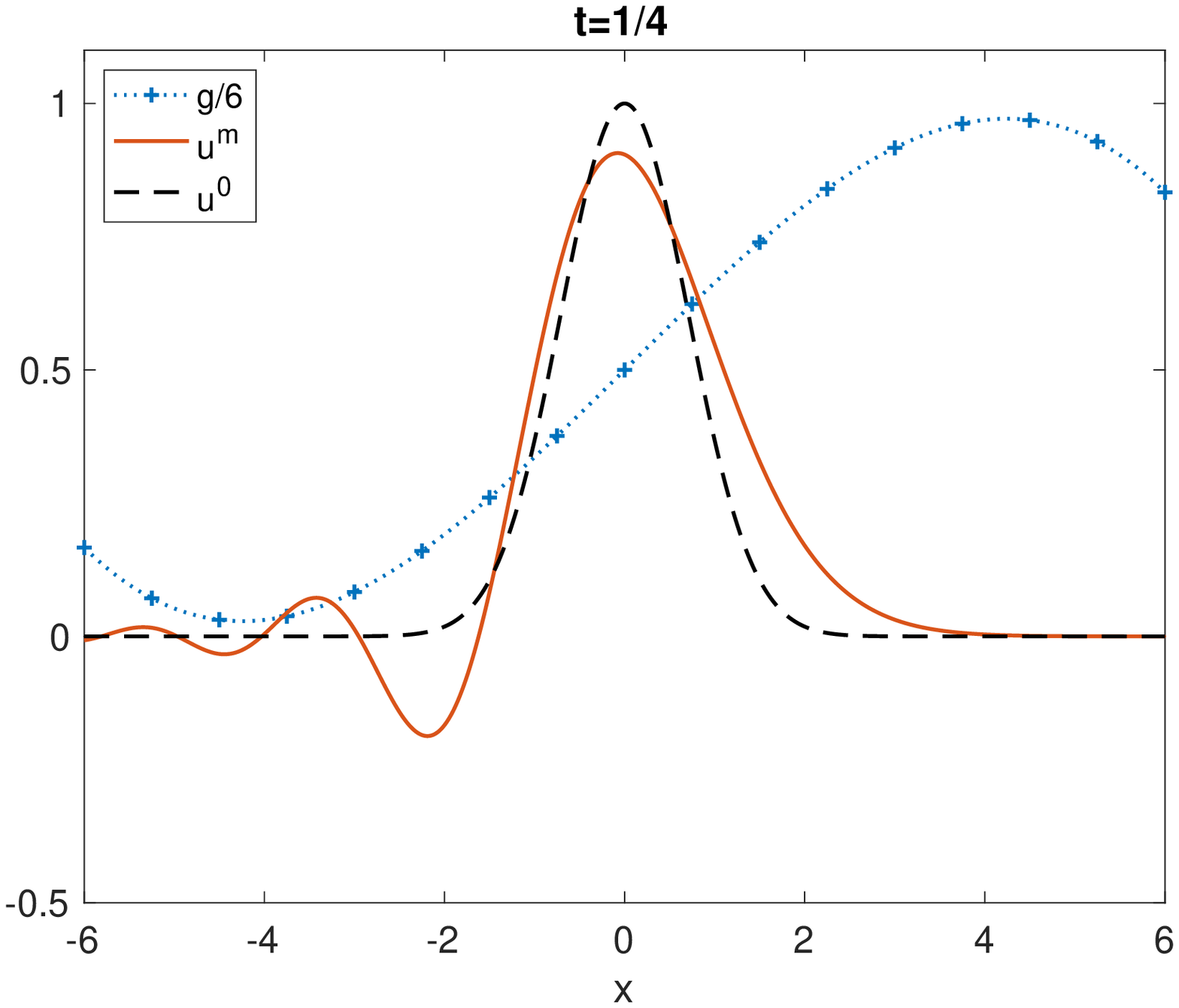}
\includegraphics[scale=.3]{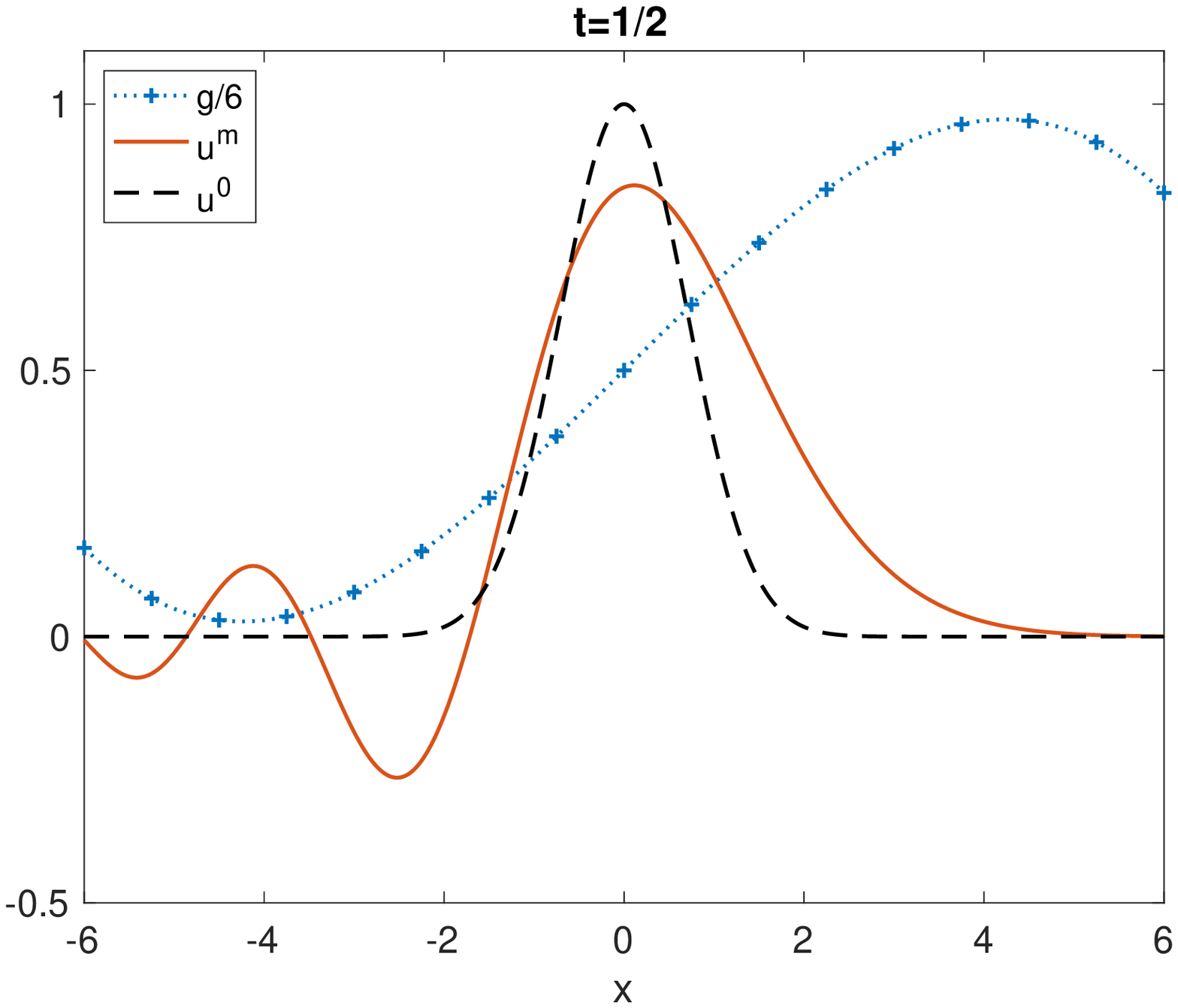}
\includegraphics[scale=.3]{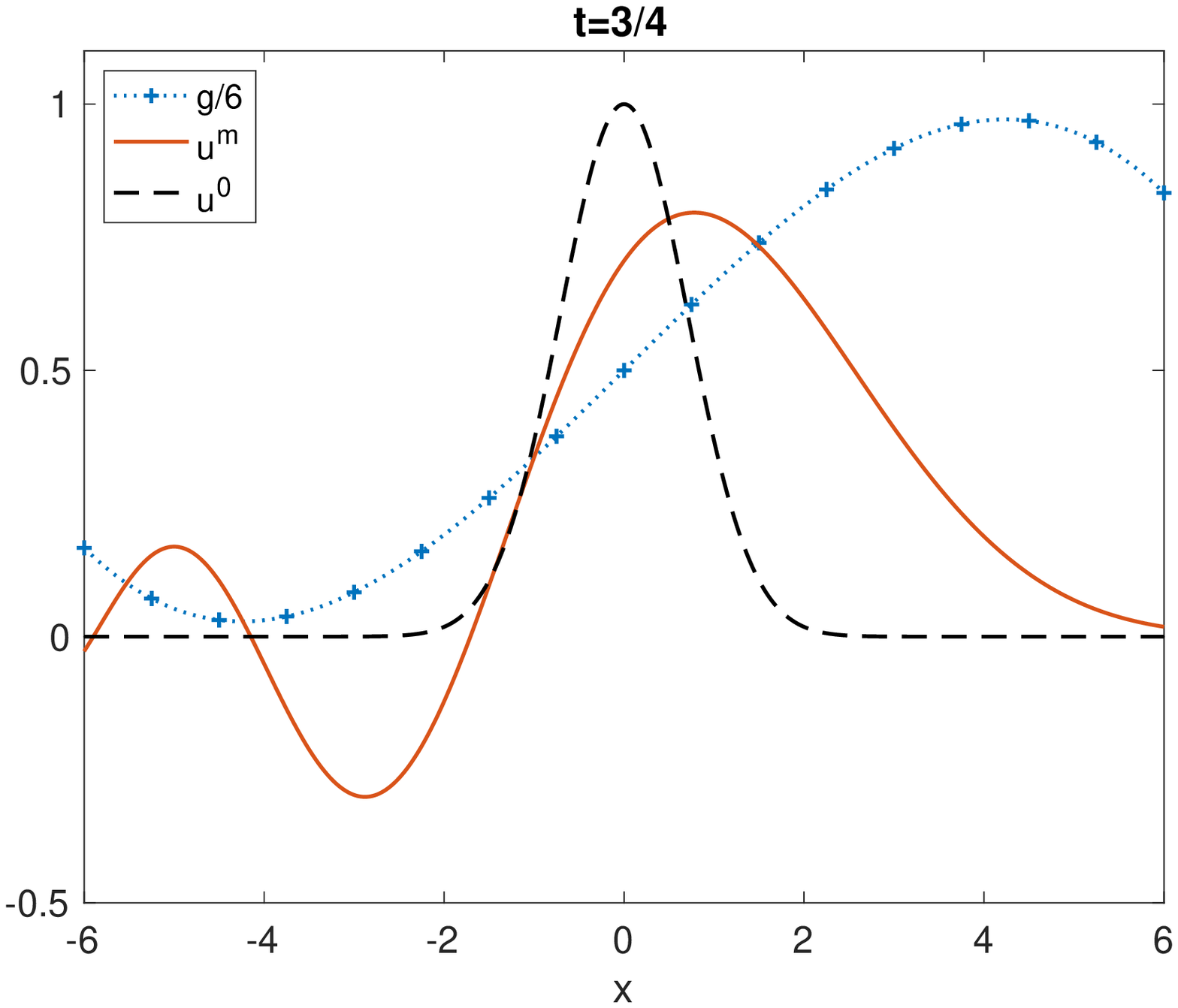}
\includegraphics[scale=.3]{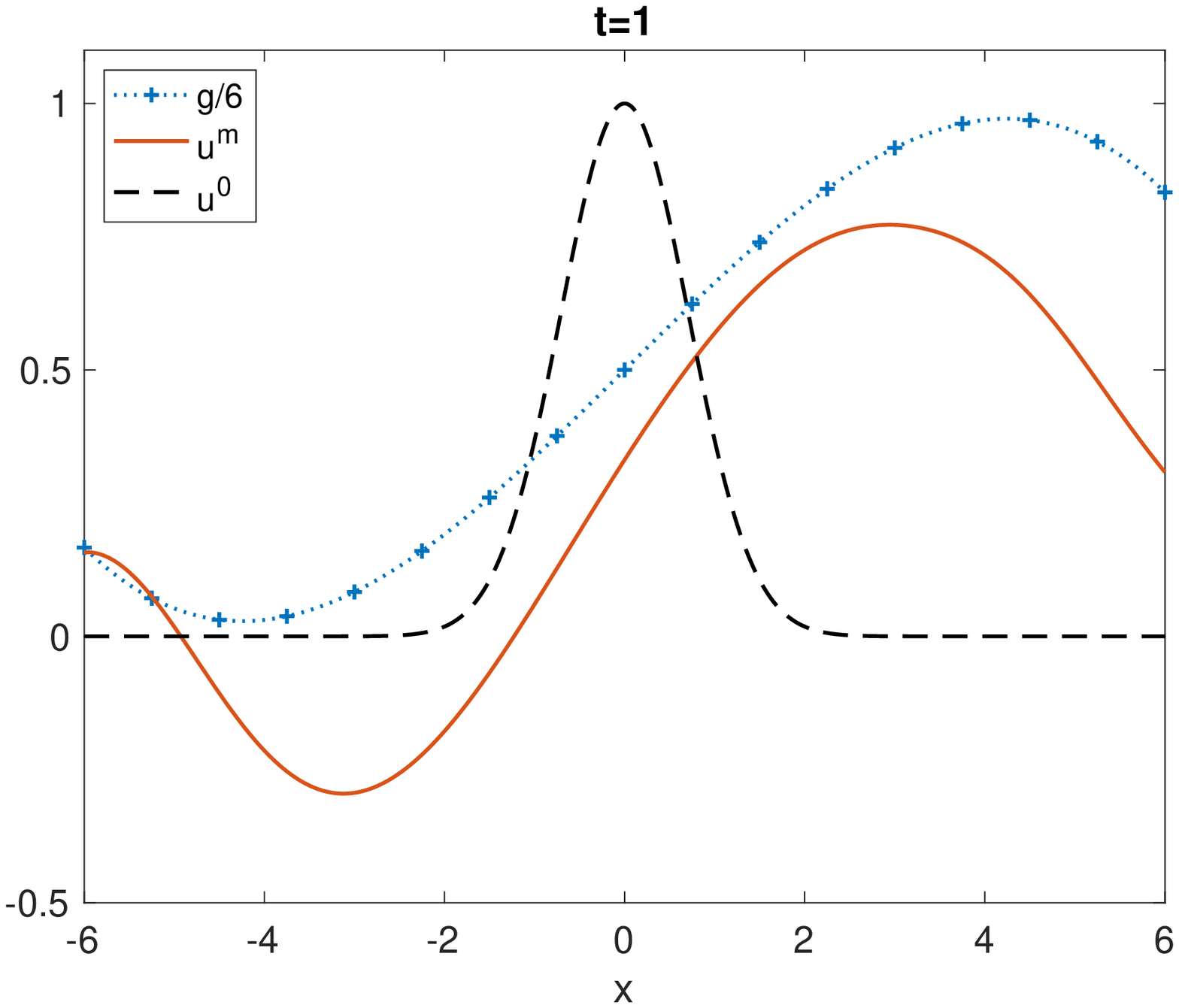}
\caption{Snapshots of the numerical solution $u^m_N$ for $g(x) = - x^3/54 + x + 3$ and $t=\frac{1}{4}$, $\frac{1}{2}$, $\frac{3}{4}$, $1$ with $\tau=2^{-12}$. The number of collocation points is set to $N=2^6$.}
\label{fig2}
\end{figure}

In Fig.~\ref{fig2} snapshots of the numerical solution $u^m_N$ for $t = \frac{1}{4}$, $\frac{1}{2}$, $\frac{3}{4}$, $1$ are shown. The solution is dragged to the right with an increasing speed. No appreciable reflections can be seen at the boundaries. Similarly to example 1, we report in Table~\ref{tab2} full discretization errors varying $N$ and $M$ with respect to a reference solution $u^m_{\text{ref}}$ computed using $N_{\text{ref}}=2^6$ and $M_{\text{ref}}=2^{12}$.
\begin{table}[!h]
\begin{center}
\small
\begin{tabular}{c c c}
$N$ & $\lVert\mathrm{err}\rVert_{\ell^2}$ & $\alpha$\Bstrut\\
\hline
$28$ & $1.2947\mathrm{e}-04$ & -- \Tstrut\\
$32$ & $2.4451\mathrm{e}-05$ & $6.9448\mathrm{e}-03$ \\
$36$ & $3.9950\mathrm{e}-06$ & $6.6605\mathrm{e}-03$ \\
$40$ & $6.0920\mathrm{e}-07$ & $6.1863\mathrm{e}-03$ \\
\hline
\end{tabular}
\hspace{2cm}
\begin{tabular}{c c c}
$M$ & $\lVert\mathrm{err}\rVert_{\ell^2}$ & $\beta$\Bstrut\\
\hline
$2^5$ & $3.7544\mathrm{e}-04$ & -- \Tstrut\\
$2^6$ & $1.0120\mathrm{e}-04$ & $1.8914$ \\
$2^7$ & $2.5507\mathrm{e}-05$ & $1.9882$ \\
$2^8$ & $6.3490\mathrm{e}-06$ & $2.0063$ \\
\hline
\end{tabular}
\caption{We present the full discretization error $\lVert\mathrm{err}\rVert_{\ell^2}$ for $g(x) = -x^3/54+x+3$.  On the left side $M$ is fixed to $2^{12}$ so that the time error is negligible w.r.t. the spatial error. On the right side $N$ is fixed to $2^6$ so that the spatial error is negligible w.r.t. the time error. In both tables, errors are obtained testing the numerical solution $u^m_N$ against a reference solution $u^m_{\text{ref}}$ using $N_{\text{ref}}=2^6$ and $M_{\text{ref}}=2^{12}$ points.}
\label{tab2}
\end{center}
\end{table}
\end{example}

\begin{example}
We consider~\eqref{eq30} with $g(x) = \mathrm{e}^{-(x+6)^2} + \mathrm{e}^{-x^2} + \mathrm{e}^{-(x-6)^2}-\frac{1}{2}$. This example is interesting because of $g$ is not polynomial and its sign alternates. The produced effects are a concentration of mass at the points $\bar{x}$ such that $g(\bar{x}) = 0$, $\partial_x g(\bar{x}) < 0$ and a thinning out where $g(\bar{x})=0$, $\partial_x g(\bar{x}) >0$. Snapshots of the numerical solution that illustrate this phenomena are shown in Fig.~\ref{fig3}. No reflections are detected at the boundaries, as expected.

\begin{figure}[t]
\centering
\includegraphics[scale=.3]{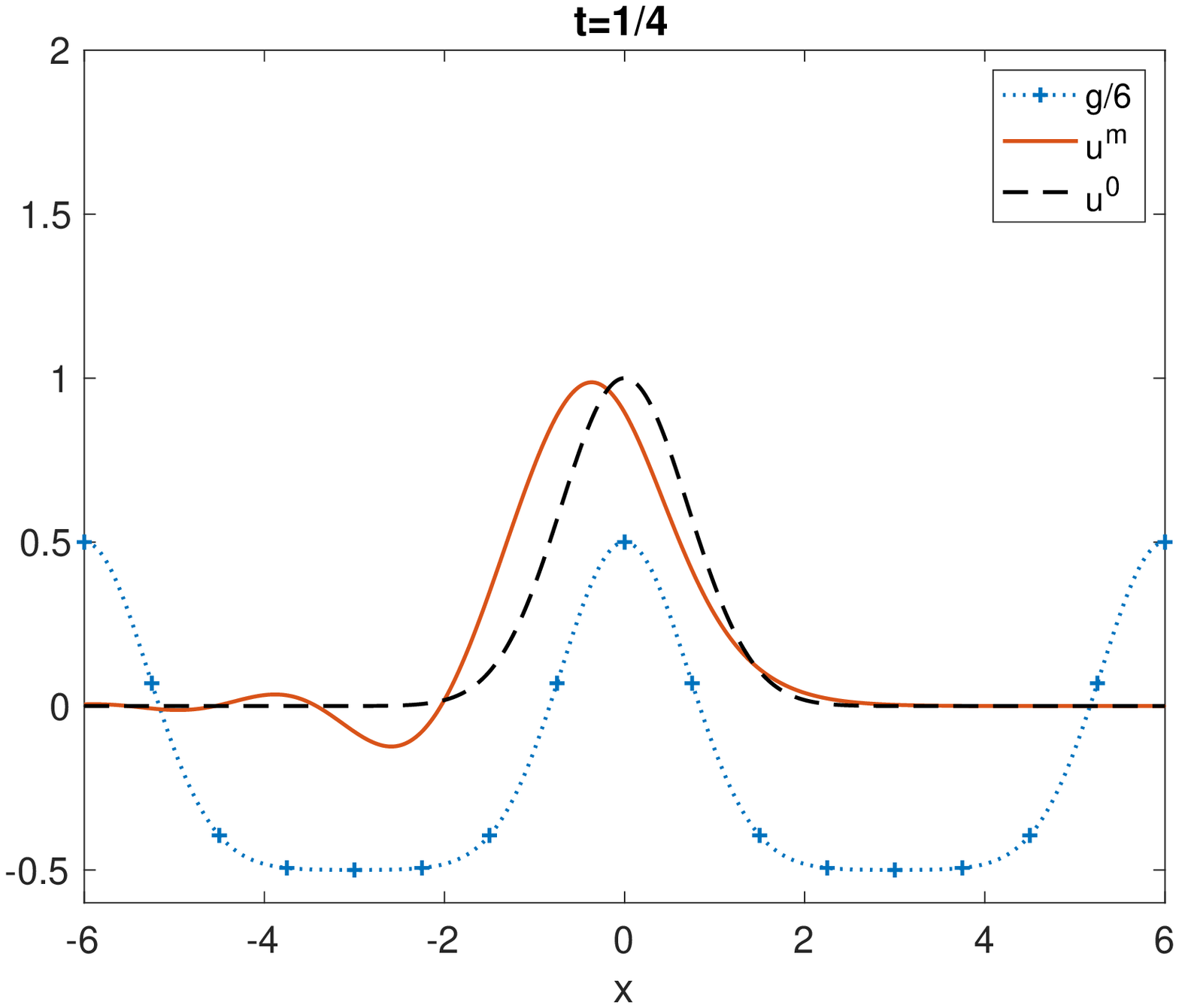}
\includegraphics[scale=.3]{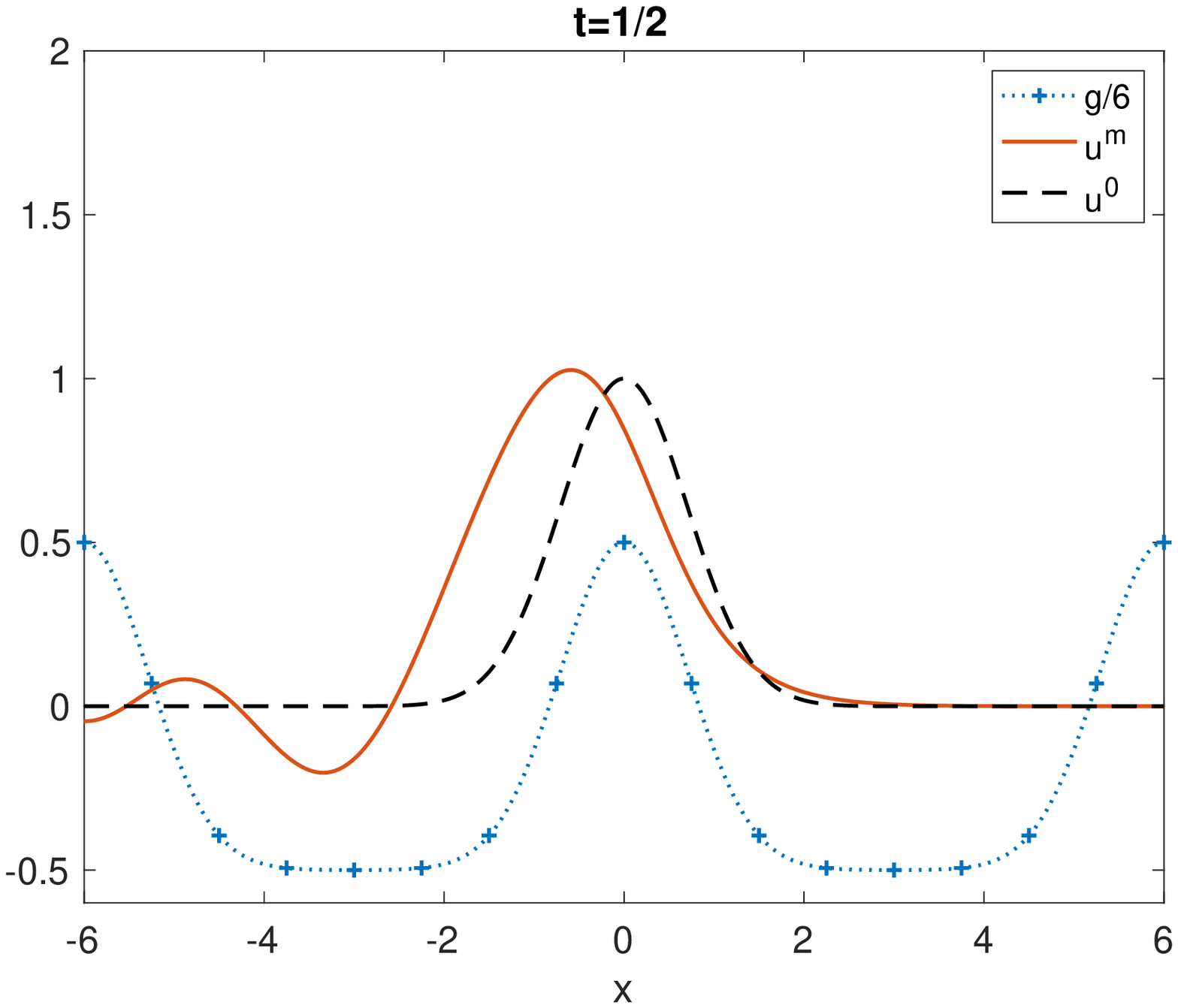}
\includegraphics[scale=.3]{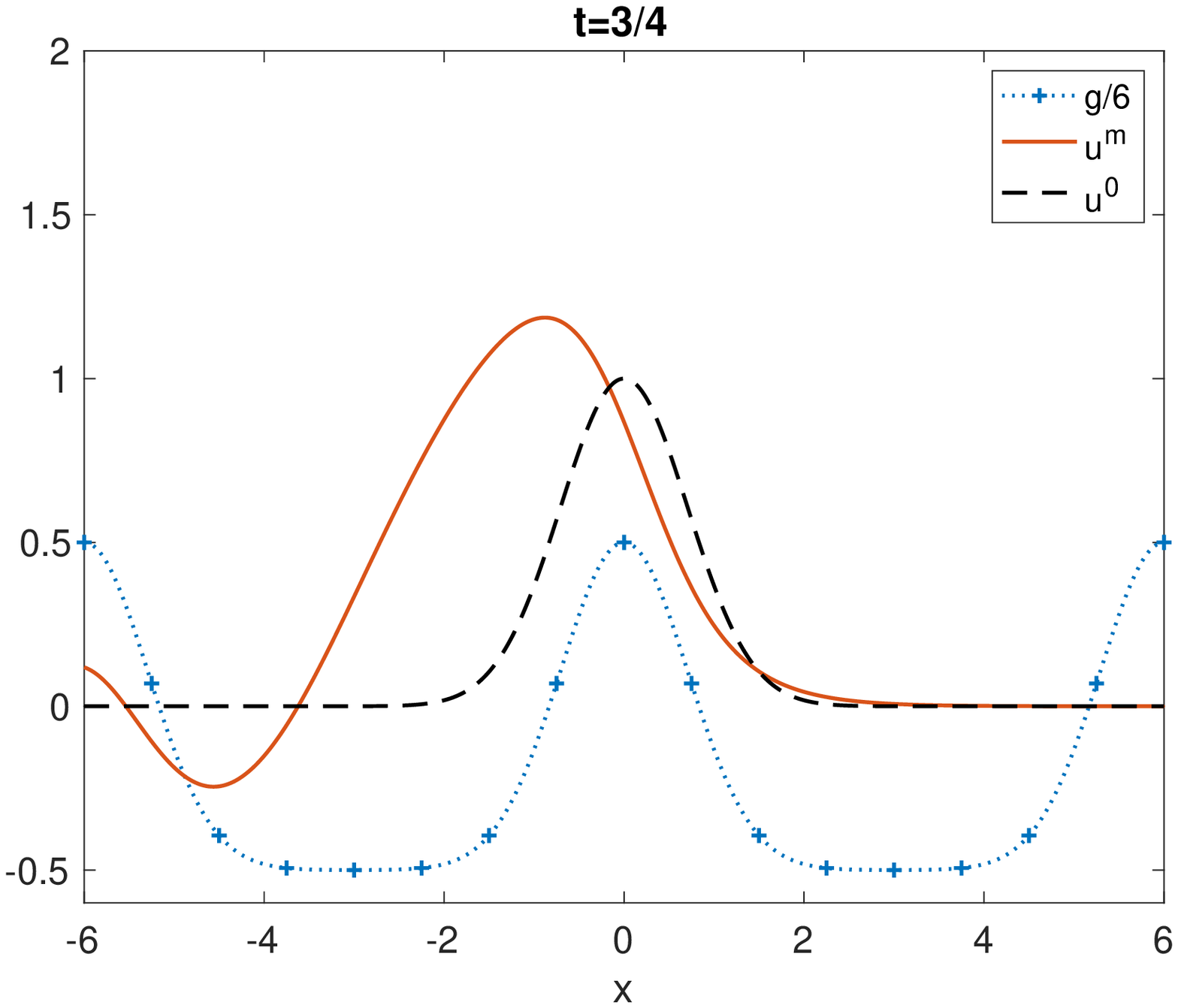}
\includegraphics[scale=.3]{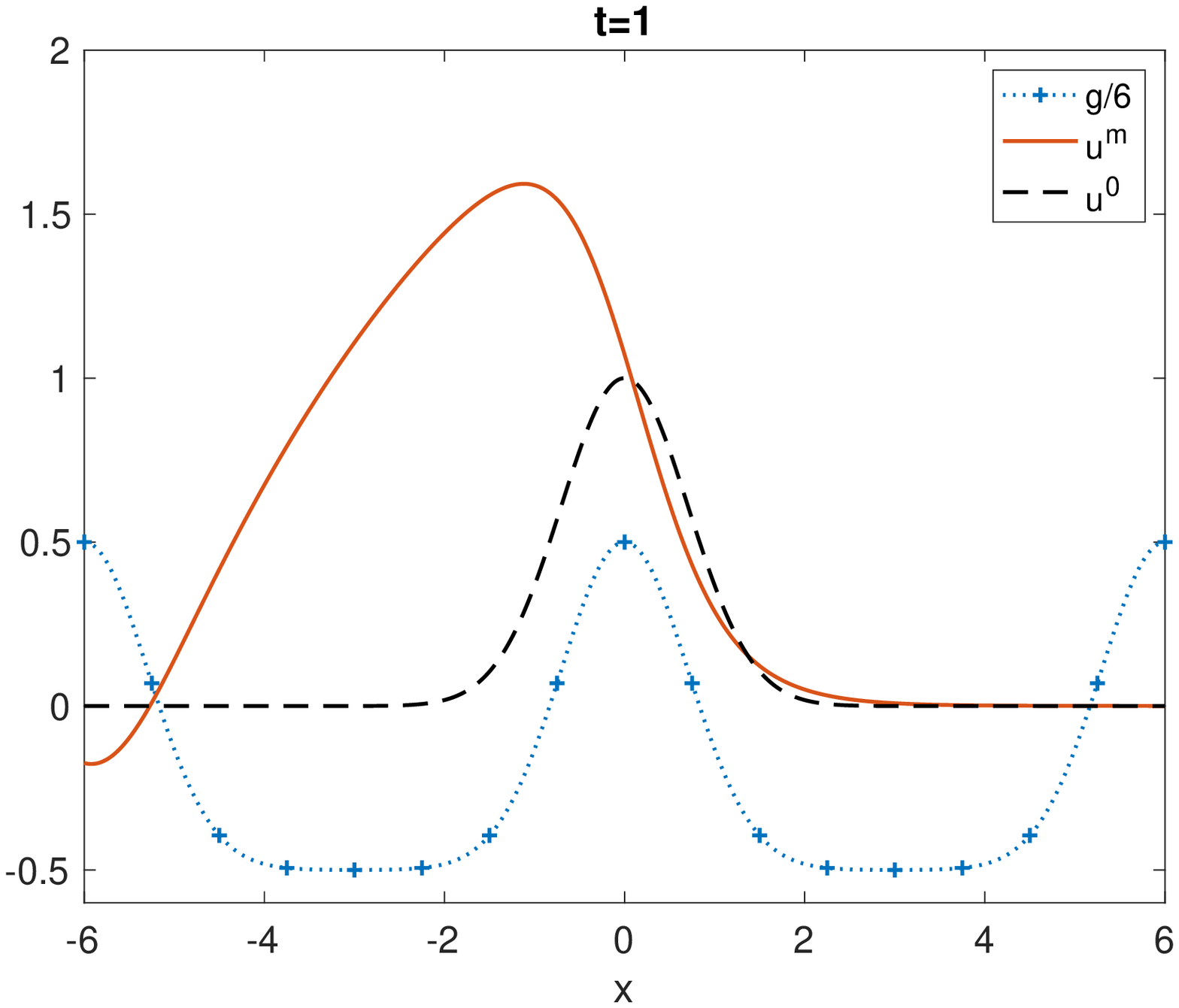}
\caption{Snapshots of the numerical solution $u^m_N$ for $g(x) = \mathrm{e}^{-(x+6)^2} + \mathrm{e}^{-x^2} + \mathrm{e}^{-(x-6)^2}-\frac{1}{2}$ and $t=\frac{1}{4}$, $\frac{1}{2}$, $\frac{3}{4}$, $1$ with $\tau=2^{-12}$.  The number of collocation points is set to $N=2^6$.}
\label{fig3}
\end{figure}

Similarly to example 2, we report in Table~\ref{tab3} the full discretization error by varying $N$ and $M$ with respect to a reference solution $u^m_{\text{ref}}$ computed using $N_{\text{ref}}=2^6$ and $M_{\text{ref}}=2^{12}$.
In Table~\ref{tab3} (left) we observe a smaller value $\alpha$ with respect to Table~\ref{tab1} and Table~\ref{tab2}. Therefore, spatial convergence is slower with respect to examples 1 and 2, but still spectral accuracy is achieved. The slower convergence rate is related to the variations of the function $g^*$, which are greater in magnitude than in examples 1 and 2.

\begin{table}[!ht]
\begin{center}
\small
\begin{tabular}{c c c}
$N$ & $\lVert\mathrm{err}\rVert_{\ell^2}$ & $\alpha$\Bstrut\\
\hline
$28$ & $5.9253\mathrm{e}-03$ & -- \Tstrut\\
$32$ & $2.6962\mathrm{e}-03$ & $3.5940\mathrm{e}-03$ \\
$36$ & $1.1380\mathrm{e}-03$ & $3.3916\mathrm{e}-03$ \\
$40$ & $4.5237\mathrm{e}-04$ & $3.1951\mathrm{e}-03$ \\
\hline
\end{tabular}
\hspace{2cm}
\begin{tabular}{c c c}
$M$ & $\lVert\mathrm{err}\rVert_{\ell^2}$ & $\beta$\Bstrut\\
\hline
$2^5$ & $2.4901\mathrm{e}-04$ & -- \Tstrut\\
$2^6$ & $7.2581\mathrm{e}-05$ & $1.7786$ \\
$2^7$ & $1.9937\mathrm{e}-05$ & $1.8642$ \\
$2^8$ & $5.0684\mathrm{e}-06$ & $1.9758$ \\
\hline
\end{tabular}
\caption{We present full error $\lVert\mathrm{err}\rVert_{\ell^2}$ for $g(x) = \mathrm{e}^{-(x+6)^2} + \mathrm{e}^{-x^2} + \mathrm{e}^{-(x-6)^2}-\frac{1}{2}$.  On the left side $M$ is fixed to $2^{12}$ so that the time error is negligible w.r.t. the spatial error. On the right side $N$ is fixed to $2^6$ so that the spatial error is negligible w.r.t. the time error. In both tables, errors are obtained testing the numerical solution $u^m_N$ against a reference solution $u^m_{\text{ref}}$ using $N_{\text{ref}}=2^6$ and $M_{\text{ref}}=2^{12}$ points.}
\label{tab3}
\end{center}
\end{table}

In Fig.~\ref{fig4} we collect error plots for examples 1,2 and 3. For all numerical tests we observe second order in time and the typical exponential convergence $\mathrm{exp}(-\alpha N^2)$, $\alpha>0$ in space.

\begin{figure}[t]
\centering
\includegraphics[scale=.38]{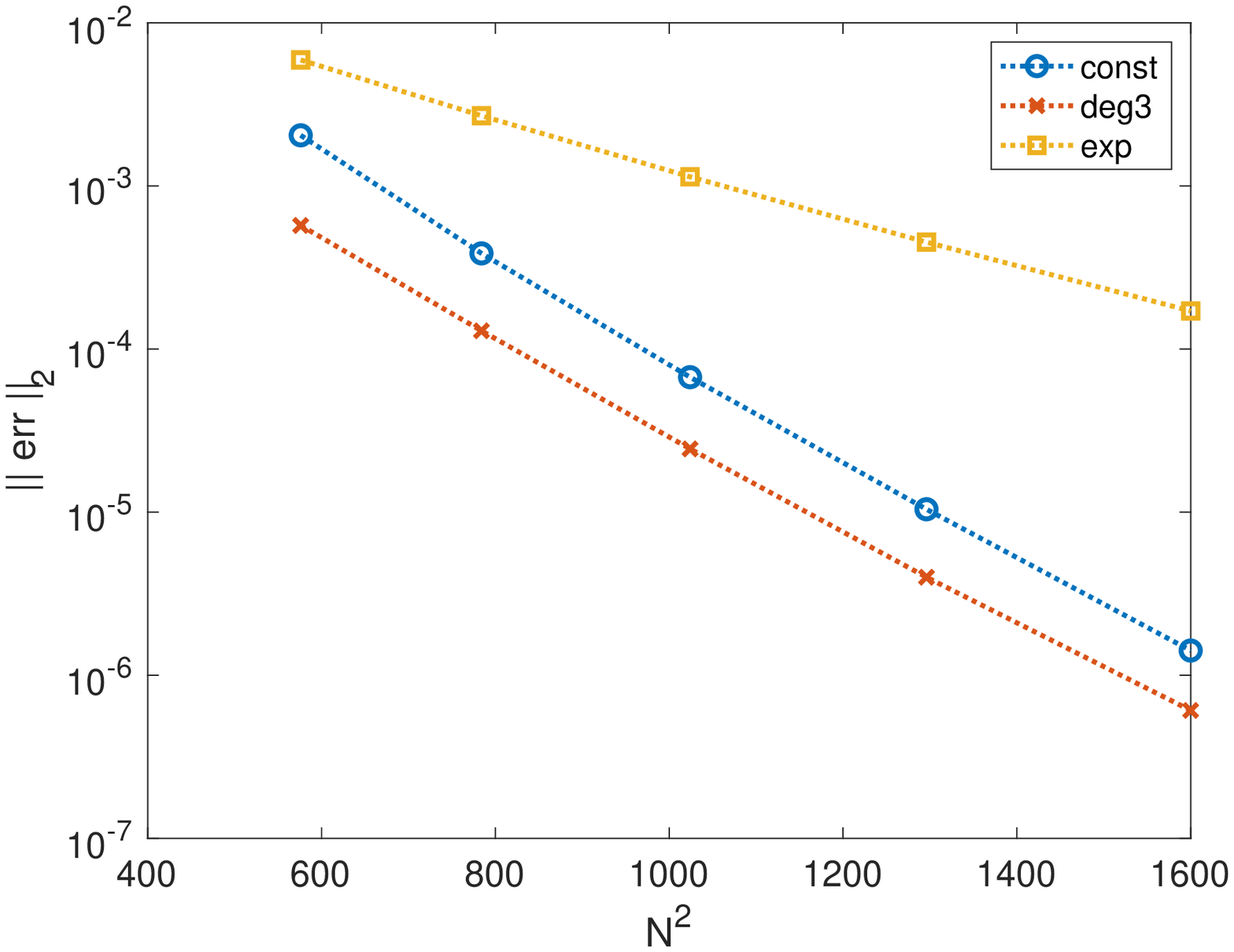}
\includegraphics[scale=.38]{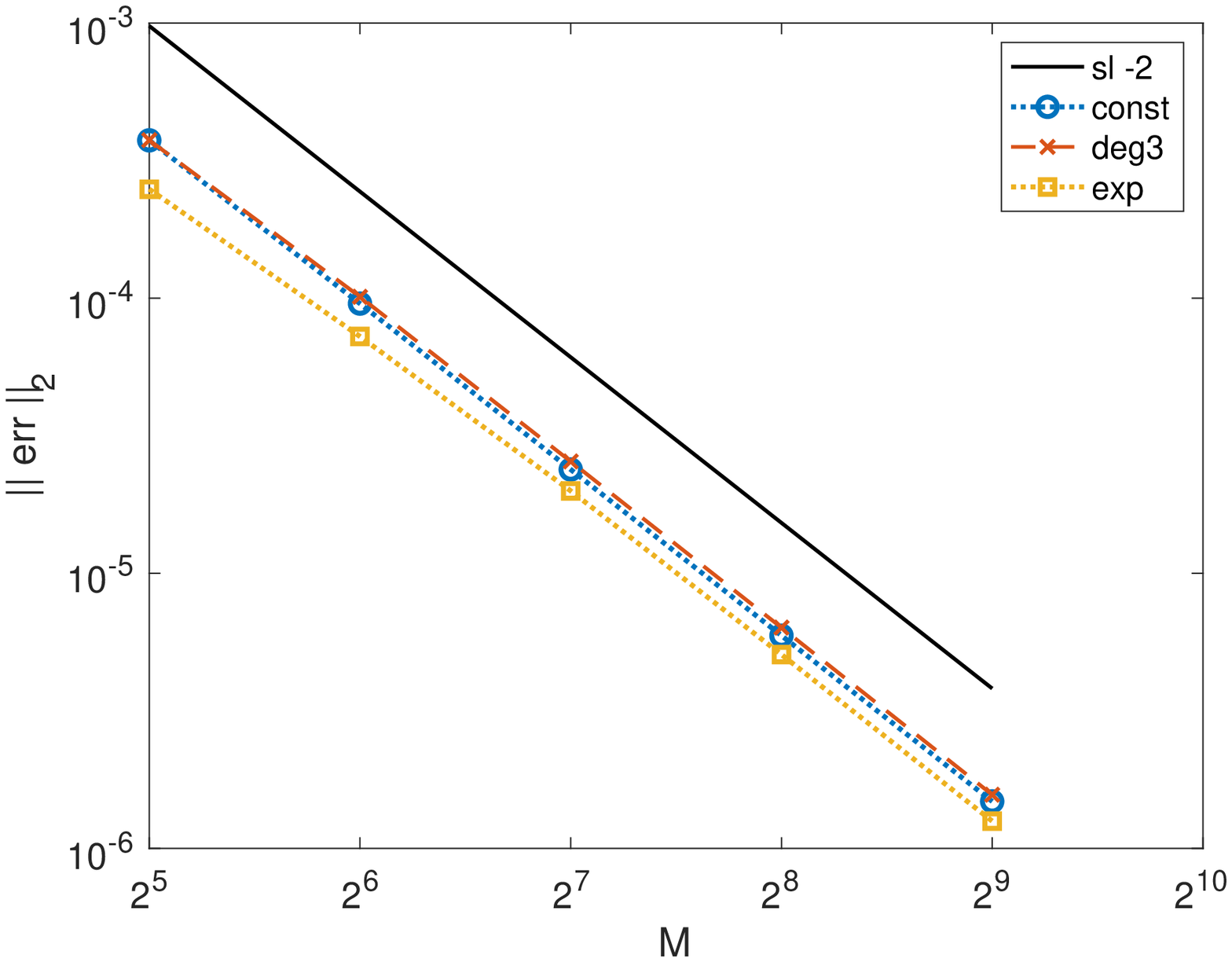}
\caption{Dotted lines show the full discretization errors $\lVert\text{err}\rVert_{\ell^2}$ between numerical solutions and a reference solutions for examples 1 (blue circles), 2 (red stars) and 3 (yellow squares).\\ \emph{(Left plot)}. On the $x$-coordinate the number of collocation points $N$, squared, varying from $24$ to $40$. On the $y$-coordinate the full discretization error $\lVert\mathrm{err}\rVert_{\ell^2}$ with $M=2^{12}$ fixed. For $N=40$ collocation points accuracy to $10^{-6}$ is achieved for examples 1 and 2, while for example 3 the accuracy is $10^{-4}$.\\
\emph{(Right plot)}. On the $x$-coordinate the number of time steps $M$ varying from $2^5$ to $2^9$. On the $y$-coordinate the full discretization error $\lVert\mathrm{err}\rVert_{\ell^2}$ with $N=2^6$ fixed. In black, a solid line of slope $-2$. Second order in time is observed for examples 1, 2 and 3.}
\label{fig4}
\end{figure}
\end{example}

The numerical experiments confirm that the proposed approach performs well in the one dimensional case. However, the extension to higher dimensions is not straightforward. Transparent boundary conditions together with the pseudo-spectral discretization become more involved to compute. This poses a real challenge and is object of future studies.

\clearpage
\bibliographystyle{siam}

\clearpage

\appendix
\section{Finding the coefficients in \eqref{eq18}}
\label{app1}
The Legendre polynomials $L_j(x)$ satisfy the following orthogonality relation
\[
(L_j,L_k) = \delta_{jk}\frac{2}{2j + 1}.
\]
Further, at $x=\pm 1$ we have
\begin{equation}
\label{eq38}
\begin{split}
L_j(\pm1) &= (\pm 1)^j,\\
\partial_x L_j(\pm 1) &= (\pm 1)^{j-1}\frac{(j)_2}{2},\\
\partial_x^2 L(\pm 1) &= (\pm1)^j \frac{(j-1)_4}{8},
\end{split}
\end{equation}
where $(j)_k = j(j+1),\dots(j+k-1)$. Inserting the dispersive basis function $\phi^d_j$ given in~\eqref{eq18} in the boundary relations of the space $V^d_N$ leads to the following linear system for $(\alpha_j,\beta_j,\gamma_j)^T$:
\[
\mathbf{A}\begin{bmatrix}
\alpha_j\\
\beta_j\\
\gamma_j
\end{bmatrix} = \mathbf{b}
\]
with $\mathbf{A}\in\R^{3\times 3}$, $\mathbf{b}\in\R^3$,
\[
\begin{split}
a_{11} &= -(g(a) + Y^0_2) + Y^0_1\frac{(j+1)_2}{2} - \frac{(j)_4}{8},\quad a_{12} = (g(a) + Y^0_2) + Y^0_1\frac{(j+2)_2}{2} - \frac{(j+1)_4}{8},\\
a_{13} &= -(g(a) + Y^0_2) + Y^0_1\frac{(j+3)_2}{2} - \frac{(j+2)_4}{8},\\
a_{21} &= -Y^0_4 + \frac{(j)_4}{8},\quad a_{22} = -Y^0_4 + \frac{(j+1)_4}{8},\quad a_{23} = -Y^0_4 + \frac{(j+2)_4}{8},\\
a_{31} &= -Y^0_3 + \frac{(j+1)_2}{2},\quad a_{32} = -Y^0_3 + \frac{(j+2)_2}{2},\quad a_{33} = -Y^0_3 + \frac{(j+3)_2}{2},\\
b_1 &= -(g(a) + Y^0_2) + Y^0_1\frac{(j)_2}{2} - \frac{(j-1)_4}{8},\\
b_2 &= Y^0_4 -\frac{(j-1)_4}{8},\\
b_3 &= Y^0_3 -\frac{(j)_2}{2}.
\end{split}
\]
Similarly, for $(\alpha_j^*$, $\beta_j^*$, $\gamma_j^*)^T$, we get
\[
\mathbf{A}^*\begin{bmatrix}
\alpha^*_j\\
\beta^*_j\\
\gamma^*_j
\end{bmatrix} = \mathbf{b}^*
\]
with
\[
\begin{split}
a^*_{11} &= (g(b) + Y^0_4) - Y^0_3\frac{(j+1)_2}{2} + \frac{(j)_4}{8},\quad a^*_{12} = (g(b) + Y^0_4) - Y^0_3\frac{(j+2)_2}{2} + \frac{(j+1)_4}{8},\\
a^*_{13} &= (g(b) + Y^0_4) - Y^0_3\frac{(j+3)_2}{2} + \frac{(j+2)_4}{8},\\
a^*_{21} &= Y^0_2 - \frac{(j)_4}{8},\quad a^*_{22} = -Y^0_2 + \frac{(j+1)_4}{8},\quad a^*_{23} = Y^0_2 - \frac{(j+2)_4}{8},\\
a^*_{31} &= -Y^0_1 + \frac{(j+1)_2}{2},\quad a^*_{32} = Y^0_1 - \frac{(j+2)_2}{2},\quad a^*_{33} = -Y^0_1 + \frac{(j+3)_2}{2},\\
b^*_1 &= -(g(b) + Y^0_4) + Y^0_3\frac{(j)_2}{2} - \frac{(j-1)_4}{8},\\
b^*_2 &= Y^0_2 -\frac{(j-1)_4}{8},\\
b^*_3 &= -Y^0_1  + \frac{(j)_2}{2}.
\end{split}
\]

\section{Inner product and discrete inner product}
\label{app2}
The entries of the mass matrix $\mathbf{M}^d$ are given by
\[\mathbf{M}^d_{kj} = \langle\phi^d_k,\psi^d_j\rangle_N^d,\quad k,j=0,\dots ,N-3.
\]
The discrete inner product is equal to the usual $L^2$ inner product for all polynomials up to degree $2N-2$. Since $\phi_k^d$ is a polynomial of degree $k+3$ and $\psi^d_j$ a polynomial of degree $j+3$ we have
\[
 \langle\phi^d_k,\psi^d_j\rangle_N^d = (\phi^d_k,\psi^d_j)\quad \text{for }k+j \leq 2N-8.
\]
This means that all entries of $\mathbf{M}^d_{kj}$ except for $(j,k)=\{(N-4,N-3),(N-3,N-4),(N-3,N-3)\}$ can be analytically pre-computed. For the last three entries the discrete inner product defined in~\eqref{eq27} must be used. This implies that for the last three entries the orthogonality relation between $\phi^d_k$ and $\psi^d_j$ might not hold for the dispersive inner product. However, the bandwidth of the matrix will not change. A similar analysis applies for the stiffness matrix $\mathbf{S}^d$. For $\mathbf{M}^{ad}$ we have
\[
\langle\phi^a_k,\psi^d_j\rangle_N^d = (\phi^a_k,\psi^d_j)\quad \text{for }k+j \leq 2N-5.
\]
The entries $(k,j) = \{(N-1,N-3),(N,N-4),(N,N-3)\}$ must be computed by using the dispersive inner product. The transition matrix $\mathbf{M}^{da}$ is given by
\[
\mathbf{M}^{da}_{kj}=\langle\phi^d_k,\psi^a_j\rangle_N^a = (\phi^d_k,\psi^a_j)\quad \text{for }k+j \leq 2N-2.
\]
Since $0\leq k\leq N-3$ and $0\leq j\leq N$ all entries can be computed analytically. A similar analysis applies for the advection mass matrix $\mathbf{M}^a$.

\section{DLT and IDLT}
\label{app3}
We recall briefly the definitions of DLT and IDLT. For more details we refer the reader to~\cite{hale15}. Given $N+1$ values $\tilde{u}_0, \tilde{u}_1,\dots \tilde{u}_{N}$  the discrete Legendre transform is defined by
\[
u_k = \sum_{n=0}^N \tilde{u}_n L_n(y_k	),\quad 0\leq k\leq N,
\]
where $y_k$ are the roots of the Legendre polynomial $L_{N+1}(y)$. The inverse discrete Legendre transform computes $\tilde{u}_0,\tilde{u}_1,\dots ,\tilde{u}_N$ for given $u_0,u_1,\dots ,u_N$. It takes the form
\[
\tilde{u}_n = \left(n +\frac{1}{2}\right)\sum_{k=0}^N w_k u_k L_n(y_k	),\quad 0\leq n\leq N,
\]
where $w_k$, $k=0,\dots N$ are the Gauss--Legendre quadrature weights. (Notice that $y_k$ and $w_k$ are the same collocation points and weights as defined in the advection inner product~\eqref{eq35}).
Both the DLT and the IDLT can be computed in $\mathcal O(N(\log N)^2/\log \log N)$ operations, see~\cite{hale15}.

For our application, let
\[
\tilde{\mathbf{s}}^{*} := \mathrm{IDLT}( g^*\partial_x u^*)\quad \text{and}\quad \tilde{\mathbf{s}}^{m+1/2} := \mathrm{IDLT}( g^*\partial_x u^{m+1/2}).
\]
Clearly, to compute $\tilde{\mathbf{s}}^{*}$ (and $\tilde{\mathbf{s}}^{m+1/2}$) we need to reconstruct $\partial_x u^*$ (and $\partial_x u^{m+1/2}$) starting from the frequency coefficients $\tilde{\mathbf{u}}^{*,a}$ (and $\tilde{\mathbf{u}}^{m+1/2,a}$). This can be done as follows. Note that we have
\[
\partial_x u^*(x) = \sum_{k=0}^{N} \tilde{u}_{k}^{*,a}\partial_x \phi^a_k(x),\quad \partial_x u^*(x) = \sum_{k=0}^{N} \widetilde{\partial_x u_{k}^{*}}^{a} \phi^a_k(x).
\]
The first relation is obtained by simply taking the derivative with respect to $x$ in~\eqref{eq20c}. The second relation comes from the fact that $\partial_x u^*$ is a polynomial of degree up to $N$ and thus it belongs to $V^{a}_N$. Therefore, it can be written as a linear combination of $V^a_N$ basis functions. Taking the advection inner product in both relations with $\psi^a_j - \psi^a_{j+2}$ for $j=0,\dots ,N,$ gives
\begin{equation}
\label{eq66}
\sum_{k=0}^{N} \tilde{u}_{k}^{*,a}\langle \partial_x \phi^a_k,\psi^a_j - \psi^a_{j+2}\rangle^a_N =  \sum_{k=0}^{N} \widetilde{\partial_x u_{k}^{*}}^{a} \langle \phi^a_k,\psi^a_j - \psi^a_{j+2}\rangle^a_N.
\end{equation}
The choice of the test functions is motivated by the fact that the resulting matrices
\[
\mathbf{F}_{kj} = \langle \partial_x \phi^a_k, \psi^a_j - \psi^a_{j+2} \rangle^a_N\quad \text{and}\quad  \mathbf{G}_{kj} = \langle \phi^a_k,\psi^a_j - \psi^a_{j+2}\rangle^a_N
\]
are both banded matrices with bandwidth two and three, respectively. In matrix form,~\eqref{eq66}  reads

\[
\mathbf{F}^T \tilde{\mathbf{u}}^{*,a} = \mathbf{G}^T \widetilde{\partial_x\mathbf{u}^{*}}^{a},
\]
from which we obtain $\widetilde{\partial_x\mathbf{u}^{*}}^{a}$ in $\mathcal{O}(N)$ operations. A similar procedure applies to the frequency coefficients of $\partial_x u^{m+1/2}$. Finally $\partial_x u^*$ and $\partial_x u^{m+1/2}$ are obtained by applying the DLT to the corresponding frequency coefficients. Summarizing, we have
\begin{align*}
\tilde{\mathbf{s}}^* &= \mathrm{IDLT}\left(g^*\cdot \mathrm{DLT}\left(\mathbf{G}^{-T}\mathbf{F}^T\tilde{\mathbf{u}}^{*,a}\right)\right),\\
\tilde{\mathbf{s}}^{m+1/2} &= \mathrm{IDLT}\left(g^*\cdot \mathrm{DLT}\left(\mathbf{G}^{-T}\mathbf{F}^T\tilde{\mathbf{u}}^{m+1/2,a}\right)\right).
\end{align*}
Thus,~\eqref{eq21} can be solved in $\mathcal O(N(\log N)^2/\log\log N)$ operations for a general function $g^*$.

\end{document}